%% file: hbdeg3_new_bilinear_constants.tex
\documentclass[reqno]{amsart}
\usepackage{amsthm, amsmath, amsfonts, amssymb, bbm}
\usepackage{enumerate, mathabx, accents}
\usepackage{color}

\input{macros.tex}
\def\dist{\text{dist}}
\def\wb1{w_{B_{\delta^{-1}}}}

\numberwithin{equation}{section}

\newcommand{\cT}{\accentset{\circ}{T}}

\newcommand{\Z}[0]{\mathbb{Z}}
\newcommand{\R}[0]{\mathbb{R}}

\newcommand{\C}[0]{\mathbb{C}}
\newcommand{\N}[0]{\mathbb{N}}

\newcommand{\ta}[0]{\theta}

\newcommand{\ld}[0]{\lambda}
\newcommand{\vep}[0]{\varepsilon}

\newcommand{\lsm}[0]{\lesssim}

\newcommand{\mc}[1]{\mathcal{#1}}
\newcommand{\ov}[1]{\overline{#1}}
\newcommand{\wt}[1]{\widetilde{#1}}
\newcommand{\st}[1]{\substack{#1}}

\newcommand{\nms}[1]{\| #1 \|}

\newcommand{\E}[0]{\mathcal{E}}

\definecolor{orange}{rgb}{1,0.5,0}

\newtheorem{thm}{Theorem}[section]
\newtheorem{lemma}[thm]{Lemma}

\newtheorem{coro}[thm]{Corollary}

\theoremstyle{remark}
\newtheorem{rem}{Remark}

\begin{document}
\title{A bilinear proof of decoupling for the cubic moment curve}
\author[S. Guo \  \ Z. Li \ \ P.-L. Yung] {Shaoming Guo, \ \ Zane Kun Li,\ \ Po-Lam Yung}
\date{\today}

\address{Shaoming Guo: Department of Mathematics, University of Wisconsin-Madison, Madison, WI-53706, USA}
\email{shaomingguo@math.wisc.edu}

\address{Zane Kun Li: Department of Mathematics, Indiana University Bloomington, Bloomington, IN-47405, USA}
\email{zkli@iu.edu}

\address{Po-Lam Yung: Department of Mathematics, The Chinese University of Hong Kong, Shatin, Hong Kong \quad \textit{and} \quad Mathematical Sciences Institute, The Australian National University, Canberra, Australia}
\email{plyung@math.cuhk.edu.hk \quad \textit{and} \quad polam.yung@anu.edu.au}


\begin{abstract}
Using a bilinear method that is inspired by the method of efficient congruencing of Wooley \cite{Woo16}, we prove a sharp decoupling inequality for the moment curve in $\R^3$.
\end{abstract}

\maketitle

\section{Introduction}

For an interval $J \subset [0, 1]$, define an extension operator
\begin{align*}
(\E_{J}g)(x) := \int_{J}g(\xi)e(x\cdot \gamma(\xi))\, d\xi
\end{align*}
where $x=(x_1, x_2, x_3)\in \R^3$, $\gamma(\xi)=(\xi, \xi^2, \xi^3)$ and $e(z) := e^{2\pi i z}$ for a real number $z\in \R$. For $\delta \in \N^{-1}$, let $P_{\delta}([0, 1])$ denote the partition of $[0, 1]$ into intervals of length $\delta$. Moreover, let $D(\delta)$ be the smallest constant such that
\begin{align}\label{decdef}
\nms{\E_{[0, 1]}g}_{L^{12}(\R^3)} \leq D(\delta)(\sum_{J \in P_{\delta}([0, 1])}\nms{\E_{J}g}_{L^{12}(\R^3)}^{4})^{1/4}
\end{align}
holds for all functions $g: [0, 1] \rightarrow \C$.
From Drury \cite{Dru85}, both sides are finite at least for smooth $g$. An inequality of this form is called an $\ell^4 L^{12}$ decoupling inequality.
Our goal will be to show the following result, which proves a sharp $\ell^{4}L^{12}$ decoupling theorem for
the moment curve $t \mapsto (t, t^2, t^3)$.
\begin{thm}\label{main}
	For every $\vep>0$ and every $\delta \in \N^{-1}$, there exists a constant $C_{\varepsilon}>0$ such that
	\begin{equation}\label{mainresult}
	D(\delta) \le C_{\varepsilon} \delta^{-\frac{1}{4} - \vep}.
	\end{equation}
    The constant $C_{\varepsilon}$ depends only on $\vep$.
\end{thm}

By a standard argument (see Section 4 of \cite{BDG16}), Theorem \ref{main} implies that
\beq\label{vinogradov_d}
\int_{[0, 1]^d} \Big|\sum_{j=1}^X e(x_1 j+x_2 j^2+\dots+x_d j^d)\Big|^{d(d+1)} dx_1 d x_2 \dots dx_d \le C_{\varepsilon} X^{\frac{d(d+1)}{2}+\varepsilon},
\endeq
for $d=3$, every positive integer $X$, every $\varepsilon>0$ and some constant $C_{\varepsilon}$ depending on $\varepsilon$. Indeed, (\ref{decdef}) implies that if $F = \sum_{J \in P_{\delta}([0,1])} F_J$ where $\widehat{F_J}$ is supported in a $\delta^3$ neighborhood of the image of $J$ under $\gamma$, then
$$
\|F\|_{L^{12}(\R^3)} \lesssim D(\delta) (\sum_{J \in P_{\delta}([0, 1])}\nms{F_J}_{L^{12}(\R^3)}^{4})^{1/4},
$$
which implies (\ref{vinogradov_d}) for $d = 3$ upon setting $\delta = 1/X$ and $$F_J(x) := \phi(x/X^3) e(x \cdot \gamma(j/X))$$ for every $J = [j/X,(j+1)/X) \in P_{\delta}([0,1])$; here $\phi$ is a Schwartz function on $\R^3$ with $\phi \geq 1$ on $[0,1]^3$, and $\widehat{\phi}$ supported on the unit ball centered at the origin.
Therefore, we recover the sharp  Vinogradov mean value estimate in $\R^3$, which was first proven by Wooley \cite{Woo16}, using the method of efficient congruencing. Later, Bourgain, Demeter and Guth \cite{BDG16} recovered \eqref{vinogradov_d} at $d=3$ and proved it for every $d\ge 4$, by using the method of decoupling. We also refer to Wooley \cite{Woo18} for a proof of \eqref{vinogradov_d} for every $d\ge 3$ using the method of efficient congruencing.

In order to prove \eqref{vinogradov_d} at $d=3$, Bourgain, Demeter and Guth first proved a stronger version of the decoupling inequality \eqref{mainresult}. To be precise, by Minkowski's inequality, the main result of \cite{BDG16} gives rise to
\beq\label{decoupling_l_2}
\nms{\E_{[0, 1]}g}_{L^{12}(\R^3)} \le C_{\varepsilon} \delta^{-\varepsilon} (\sum_{J \in P_{\delta}([0, 1])}\nms{\E_{J}g}_{L^{12}(\R^3)}^{2})^{1/2},
\endeq
for every $\varepsilon>0$. By H\"older's inequality, it is not difficult to see that \eqref{decoupling_l_2} implies \eqref{mainresult}. Moreover, by using the standard argument in Section 4 of \cite{BDG16}, \eqref{decoupling_l_2} implies \eqref{vinogradov_d} at $d=3$ just like \eqref{mainresult}. In other words, \eqref{mainresult} and \eqref{decoupling_l_2} have the same strength when deriving exponential sum estimates of the form \eqref{vinogradov_d}.

The proof of \eqref{decoupling_l_2} in \cite{BDG16} relies on multilinear methods, in particular multilinear Kakeya estimates (see for instance \cite{BCT06}, \cite{Guth15} and \cite{BBFL17}), while ours relies on a bilinear method, which involves only elementary geometric observations (see \eqref{eq:tube_intersect}).\\

The methods of efficient congruencing and decoupling use different languages: One uses the language of number theory, while the other uses purely harmonic analysis. It is a very natural and interesting question to ask whether understanding one method better could enhance our understanding of the other method.
The relation between decoupling for the parabola and efficient congruencing was studied by the second author in \cite{Li18}.
The goal of this paper is to study the relation between these two methods in the case of the cubic moment curve. In particular, we
provide a new proof of the decoupling inequality \eqref{mainresult} by using a method that is inspired by the method of efficient congruencing. Unfortunately, the new argument does not fully recover the slightly stronger decoupling inequality \eqref{decoupling_l_2}. This will be explained later in Remark \ref{remark_2_20190607} in Section \ref{20190607sub4.1}.
One significant difference between the proof here and the proof in \cite{BDG16} is that the lower dimensional input for our proof comes from a sharp ``small ball" $\ell^{4}L^{4}$ decoupling for the parabola rather than a sharp $\ell^2 L^6$ decoupling for the parabola as in \cite{BDG16}.

The authors benefited very much from the note \cite{HB15} written by Heath-Brown. In the note, Heath-Brown simplified Wooley's efficient congruencing in $\R^3$. In the current paper, we follow the structure of \cite{HB15}. We will also point out (in Section \ref{properties}) the one-to-one correspondence between main lemmas that are used in \cite{HB15} and those used in the current paper. \\

After the submission of this manuscript, the authors in collaboration with Pavel Zorin-Kranich were inspired by nested efficient congruencing \cite{Woo18} and found a proof
of sharp $\ell^{2}L^{k(k + 1)}$ decoupling for the moment curve $(t, t^2, \ldots, t^k)$ and $k \geq 2$, see \cite{GLYZK19} for more details.
This gives a much shorter and technically simpler proof of decoupling for the moment curve than the one in \cite{BDG16}.\\

{\bf Organization of paper.} In Section \ref{properties}, we will introduce the main quantities that will play crucial roles in the later proof, list the main properties of these quantities, and prove a few of them that are simple. The two key properties (Lemma \ref{bilinear1} and Lemma \ref{cor_small_ball_bi}) will be proven in Section \ref{section:tube_slab_proof} and Section \ref{sect:M2ab} respectively. After proving all these lemmas, we will use them to run an iteration argument and finish the proof of the main theorem. This step will be carried out in Section \ref{last_section}.

\bigskip

{\bf Notation.}
Give two nonnegative expressions $X$ and $Y$, by $X \lsm Y$ and $Y \gtrsim X$ we mean that there is some absolute constant $C$ such that $X \leq CY$. If $C$ depends on some additional parameters we will denote this dependence using subscripts, so for example $X \lsm_{E} Y$ means that $X \leq C_{E}Y$ for some constant $C_E$ depending on $E$.
We let $X \sim Y$ to mean that $X \lsm Y$ and $Y \lsm X$.

For a frequency interval $I$, we will use $|I|$ to denote its length. We use $P_{\delta}(I)$ to denote the partition of $I$ into intervals of length $\delta$. This implicitly assumes $|I|/\delta \in \N$. If $I=[0, 1]$, we usually
omit $[0, 1]$ and just write $P_{\delta}$ rather than $P_{\delta}([0, 1])$. For a spatial cube $B \subset \R^3$, we also use $P_{R}(B)$ to denote the partition of $B$ into cubes of side length $R$.
By $B(c, R)$, we will mean a square (or cube depending on context) centered at $c$ of side length $R$.
For a parallelpiped $T$ in $\R^3$ and a constant $c$, we let $cT$ be the dilate of $T$ where the side lengths are $c$ times larger but has the same center as $T$.

Let $E > 10^3$ be a large integer.
Let $T$ be a parallelepiped where $T = A[0,1]^3 + c$ for some $3 \times 3$ invertible matrix $A$ and some vector $c \in \R^3$. In the current paper, the columns of $A$ will be almost at right angles to each other, but can have different lengths. We write
$$
w_{T, E}(x) := (1 + |A^{-1}(x-c)|)^{-E}.
$$
for a weight that is comparable to 1 on $T$ and decays like the (non-isotropic) distance to the power $E$ outside $T$. Also write
$$
\widebar{w}_{T, E}(x) := w_{T,3E}(x) = (1 + |A^{-1}(x-c)|)^{-3E}
$$
for a weight with a faster decay.
One key property we will use about these weights is that, if $\{T\}$ is a collection of parallelepipeds that tiles a spatial cube $B \subset \R^3$, then
\begin{equation} \label{eq:wt_ineq}
\sum_T \widebar{w}_{T,E} \lesssim_E w_{B,E},
\end{equation}
with a constant that depends only on $E$.
The volume of $T$ is $|T| = |\det A|$, and we write
$$
\phi_{T, E}(x):=\frac{1}{|T|} (1 + |A^{-1}(x-c)|)^{-E}
$$
for an $L^1$ normalized version of $w_{T,E}$, that is essentially supported on $T$.

\bigskip

{\bf Acknowledgements.} Guo was supported in part by a direct grant for research from the Chinese University of Hong Kong (4053295) and NSF grant DMS-1800274. Li would like to thank his advisor Terence Tao for many discussions and support. He was supported in part by NSF grant DMS-1902763. He would also like to thank Kirsti Biggs and Sarah Peluse for some preliminary discussions related to the subject of this paper. Finally, Li would also like to thank the Department of Mathematics of the Chinese University of Hong Kong for their kind hospitality during his visit, where part of this work was done. Yung was supported in part by a General Research Fund CUHK14303817 from the Hong Kong Research Grant Council, and a direct grant for research from the Chinese University of Hong Kong (4053295). Yung would also like to thank Chin-Ho Cheung with whom he discussed efficient congruencing at a preliminary stage of this project.
Finally, the authors would like to thank the referee for detailed comments and suggestions.

\section{Main quantities and their properties}\label{properties}

For an interval $I\subset [0, 1]$, let $c_I$ denote the center of $I$. Let $\cT_I$ denote the parallelepiped that is centered at the origin, of dimension $|I|^{-1}\times |I|^{-2}\times |I|^{-3}$,
given by $$\cT_I := \{x \in \R^3 \colon |x \cdot \gamma'(c_I)| \leq |I|^{-1}, |x \cdot \gamma''(c_I)| \leq |I|^{-2}, |x \cdot \gamma'''(c_I)| \leq |I|^{-3} \}.$$
For an extremely small number $\delta$ and $\delta^{\varepsilon}\ll \nu\ll 1$ (throughout the paper we will assume that $\delta^{-1}, \nu^{-1} \in \N$), define the following two bilinear decoupling constants. For $a, b \in \N$, let $\mc{M}_{1, a, b}(\delta, \nu, E)$ and $\mc{M}_{2, a, b}(\delta, \nu, E)$ be the best constant such that
\beq
\begin{split}
 \int_{\R^3}\big( |\E_{I}g|^{2}*\phi_{\cT_I, E} \big) &\big( |\E_{I'}g|^{10}* \phi_{\cT_{I'}, E}\big) \leq \mc{M}_{1, a, b}(\delta, \nu, E)^{12}\\
& \times \bigg( \sum_{J \in P_{\delta}(I)} \|\E_J g\|_{L^{12}(\R^3)}^4 \bigg)^{1/2} \bigg( \sum_{J' \in P_{\delta}(I')} \|\E_{J'} g\|_{L^{12}(\R^3)}^4 \bigg)^{5/2}
\end{split}
\endeq
and
\beq
\begin{split}
 \int_{\R^3}\big( |\E_{I}g|^{4}*\phi_{\cT_I, E} \big) &\big( |\E_{I'}g|^{8}* \phi_{\cT_{I'}, E}\big) \leq \mc{M}_{2, a, b}(\delta, \nu, E)^{12}\\
& \times \bigg( \sum_{J \in P_{\delta}(I)} \|\E_J g\|_{L^{12}(\R^3)}^4 \bigg) \bigg( \sum_{J' \in P_{\delta}(I')} \|\E_{J'} g\|_{L^{12}(\R^3)}^4 \bigg)^{2}
\end{split}
\endeq
hold separately, for all functions $g: [0, 1] \rightarrow \C$,
and all pairs of intervals $I \in P_{\nu^a}([0, 1])$, $I' \in P_{\nu^b}([0, 1])$ with $d(I, I') \ge 2\nu$.
Note that expressions such as $|\E_{I}g|^{2} \ast \phi_{\cT_{I}, E}$ above are constant
(up to a $O_{E}(1)$ multiplicative factor) on any $|I|^{-1} \times |I|^{-2} \times |I|^{-3}$ parallelpiped
parallel to $\cT_{I}$. \\

In this section and the next two sections (but not in the last section, Section \ref{last_section}), $C_0$ is a large absolute constant whose
precise value is not important and may vary from line to line.

\begin{lemma}[Affine rescaling, cf. Lemma 1 of \cite{HB15}]\label{parab}
Let $0 < \delta < \sigma < 1$ be such that $\delta/\sigma \in \N^{-1}$.
Let $I$ be an arbitrary interval in $[0, 1]$ of length $\sigma$. Then
\begin{align*}
\nms{\E_{I}g}_{L^{12}(\R^3)} \leq D(\frac{\delta}{\sigma})(\sum_{J \in P_{\delta}(I)}\nms{\E_{J}g}_{L^{12}(\R^3)}^{4})^{1/4}
\end{align*}
for all $g: [0, 1] \rightarrow \C$.
\end{lemma}
The proof of this lemma is standard so we omit the proof
(see for example \cite[Propositon 4.1]{BD15}, \cite[Proposition 7.1]{BD17}, or \cite[Section 3.1]{Li17}).

One corollary of affine rescaling is almost multiplicativity of $D(\delta)$. This allows us to patch together the
various integrality constraints that appear throughout our argument.
\begin{coro}[Almost multiplicativity]\label{almult}
Suppose $\delta_1, \delta_2 \in \N^{-1}$, then
$$D(\delta_1\delta_2) \leq D(\delta_1)D(\delta_2).$$
\end{coro}

\begin{lemma}[Bilinear reduction, cf. Lemma 2 of \cite{HB15}]\label{bilinear}
If $\delta$ and $\nu$ were such that $\nu\delta^{-1} \in \N$, then
\begin{align*}
D(\delta) \lsm_E \nu^{-1/4}D(\frac{\delta}{\nu}) + \nu^{-1}\mc{M}_{2, 1, 1}(\delta, \nu, E).
\end{align*}
\end{lemma}
\begin{proof}
We have
\begin{align*}
\nms{\E_{[0, 1]}g}_{L^{12}(\R^3)} &= \nms{\sum_{J, J' \in P_{\nu}}\E_{J}g\ov{\E_{J'}g}}_{L^{6}(\R^3)}^{1/2}\\
&\lsm (\sum_{J \in P_{\nu}}\nms{\E_{J}g}_{L^{12}(\R^3)}^{2})^{1/2} + \nu^{-1}\max_{\st{J, J' \in P_{\nu}\\d(J, J') \geq 2\nu}}\nms{\E_{J}g\E_{J'}g}_{L^{6}(\R^3)}^{1/2}.
\end{align*}
For the first term, affine rescaling shows that it can be bounded by
\begin{align*}
D(\frac{\delta}{\nu})(\sum_{J \in P_{\nu}}(\sum_{J' \in P_{\delta}(J)}\nms{\E_{J'}g}_{L^{12}(\R^3)}^{4})^{1/2})^{1/2}.
\end{align*}
Applying H\"older in the sum over $J$, this is bounded by
$$
\nu^{-1/4} D(\frac{\delta}{\nu})(\sum_{J' \in P_{\delta}} \nms{\E_{J'}g}_{L^{12}(\R^3)}^{4})^{1/4}.
$$
This gives the first term of our desired result.
The second term follows from the observation that
\begin{align*}
\int_{\R^3} |\E_{J}g|^6 |\E_{J'}g|^6 \leq (\int_{\R^3} |\E_{J}g|^4 |\E_{J'}g|^8)^{1/2}(\int_{\R^3} |\E_{J}g|^8 |\E_{J'}g|^4)^{1/2},
\end{align*}
and the pointwise estimate
\beq\label{point_wise}
|\E_I g|^p = |\E_{I}g \ast \eta_{\cT_{I}}|^{p} \leq (|\E_{I}g| \ast |\eta_{\cT_{I}}|)^{p} \lesssim_{p} |\E_{I}g|^{p} \ast |\eta_{\cT_{I}}| \lesim_{p, E} |\E_I g|^p* \phi_{\cT_{I, E}}
\endeq
for an interval $I\subset [0, 1]$ and for every $p\ge 1$. Here $\eta_{\cT_{I}}$ is a Schwartz function whose Fourier transform is equal to 1 on a (say) $10|I| \times 10|I|^2 \times 10|I|^3$ parallelpiped containing the Fourier support of $\E_{I}g$ and decays rapidly outside this parallelpiped. Note that this is a rigorous instance of the uncertainty principle.
Combining the above two centered equations it follows that
$$
\max_{\st{J, J' \in P_{\nu}\\d(J, J') \geq 2\nu}}\nms{\E_{J}g\E_{J'}g}_{L^{6}(\R^3)}^{1/2} \lesssim_E  \max_{\st{J, J' \in P_{\nu}\\d(J, J') \geq 2\nu}} \big[ \int_{\R^3}\big( |\E_{J}g|^{4}*\phi_{\cT_J, E} \big) \big( |\E_{J'}g|^{8}* \phi_{\cT_{J'}, E}\big) \big]^{1/12} .
$$
In light of the definition of $\mc{M}_{2, 1, 1}$, this completes the proof of the lemma.
\end{proof}

\begin{lemma}[cf. Lemma 3 of \cite{HB15}]\label{lem3}
If $a$ and $b$ are integers and $\delta$ and $\nu$ were such that
$\nu^{a}\delta^{-1}, \nu^{b}\delta^{-1} \in \N$, then
$$\mc{M}_{2, a, b}(\delta, \nu, E) \lesim_{E} \mc{M}_{2, b, a}(\delta, \nu, E/C_0)^{1/3} \mc{M}_{1, a, b}(\delta, \nu, E/C_0)^{2/3},$$
for some large absolute constant $C_0$.
\end{lemma}
\begin{proof}
The proof of Lemma \ref{lem3} is essentially via H\"older's and Bernstein's inequalities.

Suppose $I \in P_{\nu^a}([0, 1])$, $I' \in P_{\nu^b}([0, 1])$ with $d(I, I') \ge 2\nu$. We first recall a version of Bernstein's inequality. Following the proof of Bernstein's inequality as in \cite[Corollary 4.3]{BD17} shows that
for every $p \geq 1$,
\begin{align}\label{bernstein}
\Big( \int_{\R^3} |\E_I g(x)|^p \phi_{\cT_I, E}(x)\, dx\Big)^{\frac 1 p} \lesim_{p, E} \int_{\R^3} |\E_I g(x)| \phi_{\cT_I, E/p}(x)\, dx.
\end{align}

Applying \eqref{bernstein} shows that there is an absolute constant $C_0$ such that
\begin{align}\label{berneq1a}
\begin{aligned}
\int_{\R^3}|\E_{I}g|^{4}\phi_{\cT_{I, E}} &= (\int_{\R^3}|\E_{I}g|^{4}\phi_{\cT_{I, E}} )^{1/3} (\int_{\R^3}|\E_{I}g|^{4}\phi_{\cT_{I, E}} )^{2/3} \\
&\lsm_{E} (\int_{\R^3}|\E_{I}g|^{4/3}\phi_{\cT_{I, E/C_0}})^{\frac{3}{4}\cdot 4 \cdot \frac{1}{3}}  (\int_{\R^3}|\E_{I}g|^{8/3}\phi_{\cT_{I, E/C_0}})^{\frac{3}{8} \cdot 4 \cdot \frac{2}{3}}.
\end{aligned}
\end{align}
Since
\begin{align*}
(|\E_{I}g|^{4} \ast \phi_{\cT_{I, E}})(x) = \int_{\R^3}|\ov{(\E_{I}g)(x - y)}|^{4}\phi_{\cT_{I, E}}(y)\, dy = \int_{\R^3}|(\E_{I}h_{x})(y)|^{4}\phi_{\cT_{I, E}}(y)\, dy
\end{align*}
where $h_{x}(\xi) = \ov{g(\xi)}e(-\xi x_1 - \xi^{2}x_2 - \xi^{3}x_3)$, from \eqref{berneq1a}
it follows that
$$
|\E_{I}g|^{4}*\phi_{\cT_{I, E}} \lesssim_E \big( |\E_{I}g|^{\frac 4 3}*\phi_{\cT_I, E/C_0} \big) \big( |\E_{I}g|^{\frac 8 3}*\phi_{\cT_I, E/C_0} \big)
$$
where we used $1 = \frac{3}{4}\cdot 4 \cdot \frac{1}{3} = \frac{3}{8} \cdot 4 \cdot \frac{2}{3}$.
Similarly,
$$
|\E_{I'}g|^{8}* \phi_{\cT_{I'}, E} \lesssim_E \big( |\E_{I'}g|^{\frac{20}{3}}* \phi_{\cT_{I'}, E/C_0}\big) \big( |\E_{I'}g|^{\frac{4}{3}}* \phi_{\cT_{I'}, E/C_0} \big)
$$
where we used $1 = \frac{3}{20}\cdot 8 \cdot \frac{5}{6} = \frac{3}{4} \cdot 8 \cdot \frac{1}{6}$.
This shows
\begin{align}\label{190514e2.2}
\begin{aligned}
& \int_{\R^3}\big( |\E_{I}g|^{4}*\phi_{\cT_I, E} \big) \big( |\E_{I'}g|^{8}* \phi_{\cT_{I'}, E}\big) \\
& \lesim_{E} \int_{\R^3}\big( |\E_{I}g|^{\frac 4 3}*\phi_{\cT_I, E/C_0} \big) \big( |\E_{I}g|^{\frac 8 3}*\phi_{\cT_I, E/C_0} \big)\times\\
& \hspace{1.5in}\big( |\E_{I'}g|^{\frac{20}{3}}* \phi_{\cT_{I'}, E/C_0}\big) \big( |\E_{I'}g|^{\frac{4}{3}}* \phi_{\cT_{I'}, E/C_0}\big).
\end{aligned}
\end{align}

By convexity and H\"older, the last display can be bounded by
\begin{align}\label{190514e2.3}
\begin{aligned}
\Big( \int_{\R^3} \big( |\E_{I}g|^{2}*\phi_{\cT_I, E/C_0} \big) \big( &|\E_{I'}g|^{10}* \phi_{\cT_{I'}, E/C_0}\big) \Big)^{\frac  2 3}\times\\
& \Big( \int_{\R^3}\big( |\E_{I}g|^{8}*\phi_{\cT_I, E/C_0} \big) \big( |\E_{I'}g|^{4}* \phi_{\cT_{I'}, E/C_0}\big)\Big)^{\frac 1 3}
\end{aligned}
\end{align}
Recalling the definitions of $\mc{M}_{2, a, b}$ and $\mc{M}_{1, a, b}$, this finishes the proof of the lemma.
\end{proof}

\begin{lemma}[cf. Lemma 4 of \cite{HB15}]\label{lem4}
If $a$ and $b$ are integers and $\delta$ and $\nu$ were such that $\nu^{a}\delta^{-1}, \nu^{b}\delta^{-1} \in \N$, then
$$\mc{M}_{1, a, b}(\delta, \nu, E) \lsm_{E} \mc{M}_{2, b, a}(\delta, \nu, E/C_0)^{1/4}D(\frac{\delta}{\nu^b})^{3/4},$$
for some large absolute constant $C_0$.
\end{lemma}
\begin{proof}
Suppose $I \in P_{\nu^a}([0, 1])$, $I' \in P_{\nu^b}([0, 1])$ with $d(I, I') \ge 2\nu$. We start with an estimate that is similar to \eqref{190514e2.2} and \eqref{190514e2.3}:
\begin{align*}
\begin{split}
& \int_{\R^3}\big( |\E_{I}g|^{2}*\phi_{\cT_I, E} \big) \big( |\E_{I'}g|^{10}* \phi_{\cT_{I'}, E}\big) \\
& \lesim_{E} \int_{\R^3}\big( |\E_{I}g|^{2}*\phi_{\cT_I, E} \big) \big( |\E_{I'}g|* \phi_{\cT_{I'}, E/C_0}\big) \big( |\E_{I'}g|^{9}* \phi_{\cT_{I'}, E/C_0}\big)\\
& \lesim_{E} \Big( \int_{\R^3}\big( |\E_{I}g|^{8}*\phi_{\cT_I, E} \big) \big( |\E_{I'}g|^{4}* \phi_{\cT_{I'}, E/C_0}\big)\Big)^{\frac 1 4} \Big( \int_{\R^3} |\E_{I'}g|^{12}* \phi_{\cT_{I'}, E/C_0}\Big)^{\frac 3 4}.
\end{split}
\end{align*}
Since $\phi_{\cT_{I'}, E/C_0}$ is $L^{1}$-normalized,
\begin{align*}
\int_{\R^3} \big( |\E_{I'}g|^{12}* \phi_{\cT_{I'}, E/C_0}\big) \lesim \int_{\R^3}  |\E_{I'}g|^{12}.
\end{align*}
We finish the proof by applying affine rescaling.
\end{proof}

The proofs of the following two lemmas will be given in Sections \ref{section:tube_slab_proof} and \ref{sect:M2ab}.
\begin{lemma}[cf. Lemma 5 of \cite{HB15}]\label{bilinear1}
Let $a$ and $b$ be integers such that $1 \leq a \leq 3b$. Suppose $\delta$ and $\nu$ were such that $\nu^{3b}\delta^{-1} \in \N$.
Then
\begin{equation} \label{Lemma6'}
\mc{M}_{1,a,b}(\delta, \nu, E) \lesssim_{a, b, E} \nu^{-\frac{1}{24}(3b - a) - C_0} \mc{M}_{1, 3b, b}(\delta, \nu, E/C_0)
\end{equation}
for some absolute constant $C_0$.
\end{lemma}

\begin{lemma}[cf. Lemma 6 of \cite{HB15}]\label{cor_small_ball_bi}
Let $a$ and $b$ be integers such that $1 \leq a \leq b$. Suppose $\delta$ and $\nu$ were such that $\nu^{2b - a}\delta^{-1} \in \N$
and $\nu \in 2^{-2^{\mathbb{N}}} \cap (0,1/1000)$. Then for every $\vep > 0$,
\begin{align*}
\mc{M}_{2, a, b}(\delta, \nu, E) \lsm_{\vep, E} \nu^{-\frac{1}{6}(1+\vep)(b - a) - C_0}\mc{M}_{2, 2b - a, b}(\delta, \nu, E/C_0),
\end{align*}
for some absolute constant $C_0$.
\end{lemma}

\section{The first bilinear constant $\mc{M}_{1, a,b}$}\label{section:tube_slab_proof}
We break the proof of \eqref{Lemma6'} into the following three different lemmas.

\begin{lemma}[$\ell^2 L^2$ decoupling] \label{lem:b}
	If $1 \leq a \leq b$, then for any pair of frequency intervals $I, I' \subset [0,1]$ with $|I| = \nu^a$, $|I'| = \nu^b$, $d(I,I') \geq 2\nu$, we have

\beq\label{190505e3.1}
\begin{split}
& \int_{\R^3} \big( |\E_{I}g|^{2}*\phi_{\cT_I, E} \big) \big( |\E_{I'}g|^{10}* \phi_{\cT_{I'}, E}\big)\\
& \lesim_E \sum_{J\in P_{\nu^b}(I)} \int_{\R^3} \big( |\E_{J}g|^{2}*\phi_{\cT_J, E/C_0} \big) \big( |\E_{I'}g|^{10}* \phi_{\cT_{I'}, E/C_0}\big)
\end{split}
\endeq
for large enough $E$ and for some absolute constant $C_0$.
\end{lemma}

\begin{lemma}[Ball inflation] \label{lem:2a}
	If $b \leq a \leq 2b$, then for any pair of frequency intervals $I, I' \subset [0,1]$ with $|I| = \nu^a$, $|I'| = \nu^b$, $d(I,I') \geq 2\nu$, we have

\beq\label{190505e3.2}
\begin{split}
& \int_{\R^3} \big( |\E_{I}g|^{2}*\phi_{\cT_I, E} \big) \big( |\E_{I'}g|^{10}* \phi_{\cT_{I'}, E}\big)\\
& \lesim_E \nu^{-C_0} \sum_{J\in P_{\nu^{2b}}(I)} \int_{\R^3} \big( |\E_{J}g|^{2}*\phi_{\cT_J, E/C_0} \big) \big( |\E_{I'}g|^{10}* \phi_{\cT_{I'}, E/C_0}\big)
\end{split}
\endeq
for large enough $E$ and for some absolute constant $C_0$.
\end{lemma}

\begin{lemma}[Ball inflation] \label{lem:3b}
	If $2b \leq a \leq 3b$, then for any pair of frequency intervals $I, I' \subset [0,1]$ with $|I| = \nu^a$, $|I'| = \nu^b$, $d(I,I') \geq 2\nu$, we have
	\beq\label{190505e3.3}
\begin{split}
& \int_{\R^3} \big( |\E_{I}g|^{2}*\phi_{\cT_I, E} \big) \big( |\E_{I'}g|^{10}* \phi_{\cT_{I'}, E}\big)\\
& \lesim_E \nu^{-C_0} \sum_{J\in P_{\nu^{3b}}(I)} \int_{\R^3} \big( |\E_{J}g|^{2}*\phi_{\cT_J, E/C_0} \big) \big( |\E_{I'}g|^{10}* \phi_{\cT_{I'}, E/C_0}\big)
\end{split}
\endeq
for large enough $E$ and for some absolute constant $C_0$.
\end{lemma}

Combining the three lemmas, we see that if $1 \leq a \leq 3b$ and $\nu^{3b}\delta^{-1} \in \N$, then for any pair of frequency intervals $I, I' \subset [0,1]$ with $|I| = \nu^a$, $|I'| = \nu^b$, $d(I,I') \geq 2\nu$,
we have
\begin{align*}
\begin{split}
 \int_{\R^3} &\big( |\E_{I}g|^{2}*\phi_{\cT_I, E} \big) \big( |\E_{I'}g|^{10}* \phi_{\cT_{I'}, E}\big)\\
& \lesim_E \nu^{-C_0} \sum_{J\in P_{\nu^{3b}}(I)} \int_{\R^3} \big( |\E_{J}g|^{2}*\phi_{\cT_J, E/C_0} \big) \big( |\E_{I'}g|^{10}* \phi_{\cT_{I'}, E/C_0}\big)
\end{split}
\end{align*}
which is further bounded by
\[
\begin{split}
\lesssim_E & \nu^{-C_0} \mc{M}_{1,3b,b}(\delta,\nu, E/C_0)^{12} \\
&\quad\quad \sum_{J \in P_{\nu^{3b}}(I)} \bigg( \sum_{J'' \in P_{\delta}(J)} \|\E_{J''} g\|_{L^{12}}^4 \bigg)^{1/2} \bigg( \sum_{J' \in P_{\delta}(I')} \|\E_{J'} g\|_{L^{12}}^4 \bigg)^{5/2} \\
\lesssim_E & \nu^{-C_0-(3b-a)(1-\frac{2}{4})} \mc{M}_{1,3b,b}(\delta,\nu, E/C_0)^{12} \\
&\quad\quad \bigg( \sum_{J \in P_{\delta}(I)} \|\E_J g\|_{L^{12}}^4 \bigg)^{1/2} \bigg( \sum_{J' \in P_{\delta}(I')} \|\E_{J'} g\|_{L^{12}}^4 \bigg)^{5/2}.
\end{split}
\]
It is clear that \eqref{Lemma6'} now follows from the definition of $\mc{M}_{1,a,b}(\delta,\nu, \cdot )$.\\

First we prove a small technical lemma that will be used in the proof of Lemma \ref{lem:b}.
\begin{lemma}\label{eq39}
For $J \subset I \subset [0, 1]$,
$$|\E_{J}g|^{2} \ast \phi_{\cT_I, E} \lsm_{E} |\E_{J}g|^{2} \ast \phi_{\cT_J, E/C_0}$$
for some sufficiently large $C_0$.
\end{lemma}
\begin{proof}
First it suffices to instead show that for $J \subset I \subset [0, 1]$, we have
\begin{align}\label{e3.9}
\nms{\E_{J}g}_{L^{2}(\phi_{\cT_I, E})}^{2} \lsm \nms{\E_{J}g}_{L^{2}(\phi_{\cT_J, E/C_0})}^{2}.
\end{align}
Suppose $|J| = 1/R'$ and $|I| = 1/R$ with $R' \geq R$. It suffices to only show the case when $I = [0, 1/R]$.
Since $J \subset [0, 1/R]$, the angle between $\cT_I$ and $\cT_J$ is $O(1/R)$.
Therefore $\cT_I$ is contained in a rectangle that is a $O(1)$
dilation of $\cT_I$ but pointing in the same direction as $\cT_J$. Furthermore this dilate of
$\cT_I$ is contained in a $O(1)$ dilation of $\cT_J$. Thus there exists a sufficiently
large absolute constant $C$ such that $\cT_I \subset C\cT_J$. The same reasoning gives that for $k \geq 0$,
$2^{k}\cT_I \subset C2^{k}\cT_J$ where $C$ is an absolute constant.

We first prove an unweighted version of \eqref{e3.9}. Fix $k \geq 0$. Then
\begin{align*}
\frac{1}{|2^k \cT_I|}\|E_{J}g\|^2_{L^{2}(2^{k}\cT_I)} &\leq \nms{\E_{J}g}_{L^{\infty}(C2^{k}\cT_J)}^{2}\\
& \lsm_{E} 2^{3k}\nms{\E_{J}g}_{L^{2}(\phi_{C2^{k}\cT_J, E/100})}^{2} \lsm_{E} 2^{3k}\nms{\E_{J}g}_{L^{2}(\phi_{2^{k}\cT_J, E/100})}^{2}.
\end{align*}
Next, observe that
\begin{align*}
\phi_{\cT_I, E}(x) \lsm_{E} \sum_{k \geq 0}2^{-k(E-3)}\frac{1}{|2^{k}\cT_I|}\mathbbm{1}_{2^{k}\cT_I}(x).
\end{align*}
Therefore
\begin{align*}
\nms{\E_{J}g}_{L^{2}(\phi_{\cT_I, E})}^{2} &\lsm \sum_{k \geq 0}2^{-k(E-3)}\int_{\R^3}|(\E_{J}g)(x)|^{2}\frac{1}{|2^{k}\cT_I|}\mathbbm{1}_{2^{k}\cT_I}(x)\, dx\\
&\lsm_{E}\sum_{k \geq 0}2^{-k(E - 6)}\int_{\R^3}|(\E_{J}g)(x)|^{2}\phi_{2^{k}\cT_J, E/100}(x)\, dx\\
&\lsm_{E}\sum_{k \geq 0}2^{-k(E - 6)}\int_{\R^3}|(\E_{J}g)(x)|^{2}2^{-3k + 3kE/100}\phi_{\cT_J, E/100}(x)\, dx\\
&\lsm_{E}\nms{\E_{J}g}_{L^{2}(\phi_{\cT_J, E/100})}^{2}.
\end{align*}
This completes the proof of Lemma \ref{eq39}.
\end{proof}

We now move on to the proofs of Lemmas \ref{lem:b}--\ref{lem:3b}.

\begin{proof}[Proof of Lemma~\ref{lem:b}]
Let $\{\Box\}$ be a partition of $\R^3$ into cubes of side length $\nu^{-b}$. We write the left hand side of \eqref{190505e3.1} as
\begin{align*}
\sum_{\Box} \int_{\Box} \big( |\E_{I}g|^{2}*\phi_{\cT_I, E} \big) \big( |\E_{I'}g|^{10}* \phi_{\cT_{I'}, E}\big).
\end{align*}
We bound the above expression by
\beq\label{190514e3.6}
\sum_{\Box} \big(\sup_{x\in \Box} |\E_{I'}g|^{10}* \phi_{\cT_{I'}, E}(x)\big) \int_{\Box} \big( |\E_{I}g|^{2}*\phi_{\cT_I, E} \big).
\endeq
We write the latter factor as
\beq\label{190505e3.7}
\int_{\R^3} \big[\int_{\Box_y} |\E_I g(x)|^2dx\big] \phi_{\cT_{I}, E}(y) dy,
\endeq
where $\Box_y:=\Box-y$. By $L^2$ orthogonality (see for instance Appendix of \cite{GZo18}), we have
\beq\label{190514e3.8}
\begin{split}
& \eqref{190505e3.7} \lesim_E \sum_{J\in P_{\nu^b}(I)} \int_{\R^3} \Big[\int_{\R^3} |\E_J g(x)|^2 w_{\Box_y, C_0 E}(x)dx\Big] \phi_{\cT_{I}, E}(y) dy\\
& \lesim_E \sum_{J\in P_{\nu^b}(I)} \int_{\R^3} \Big[\int_{\R^3} |\E_J g(x-y)|^2 w_{\Box, C_0 E}(x)dx\Big] \phi_{\cT_{I}, E}(y) dy\\
& \lesim_E \sum_{J\in P_{\nu^b}(I)} \int_{\R^3} \big(|\E_J g|^2*\phi_{\cT_{J}, E/C_0} \big) w_{\Box, C_0 E},
\end{split}
\endeq
where in the last step we have used Lemma \ref{eq39}. This, combined with the definition of the weight $w_{\Box, C_0 E}$, implies that \eqref{190514e3.6} can be bounded by
\beq\label{190514e3.10}
\sum_{\Box} \sum_{J\in P_{\nu^b}(I)} \sum_{\kappa\in \Z^3} (1+|\kappa|)^{-C_0 E} \big(\sup_{x\in \Box} |\E_{I'}g|^{10}* \phi_{\cT_{I'}, E}(x)\big) \int_{\Box_{\nu^{-b}\kappa}} \big(|\E_J g|^2*\phi_{\cT_{J}, E/C_0} \big).
\endeq
In the end, we just need to observe that
\beq\label{190514e3.11}
\sup_{x\in \Box_{\nu^{-b}\kappa}} |\E_{J}g|^{2}*\phi_{\cT_J, E/C_0}(x)\lsm |\kappa|^{E/C_0} \inf_{x'\in \Box} |\E_{J}g|^{2}*\phi_{\cT_J, E/C_0}(x'),
\endeq
and
\beq\label{190514e3.11a}
\sup_{x \in \Box}|\E_{I'}g|^{10} \ast \phi_{\cT_{I'}, E}(x) \sim_{E} \inf_{x \in \Box}|\E_{I'}g|^{10} \ast \phi_{\cT_{I'}, E}(x)
\endeq
both of which follow from the definition of the weight $\phi$.
Inserting \eqref{190514e3.11} and \eqref{190514e3.11a} into \eqref{190514e3.10} and using that $|\Box_{\nu^{-b}\kappa}| = |\Box|$ shows that \eqref{190514e3.10} is bounded by
\begin{align*}
&(\sum_{\kappa \in \Z^3}(1 + |\kappa|)^{-C_0 E}|\kappa|^{E/C_0})\times\\
&\hspace{0.5in}\sum_{\Box}\sum_{J \in P_{\nu^b}(I)}(\inf_{x'\in \Box} |\E_{J}g|^{2}*\phi_{\cT_J, E/C_0}(x'))(\inf_{x \in \Box}|\E_{I'}g|^{10} \ast \phi_{\cT_{I'}, E}(x))|\Box|\\
&\lsm \sum_{\Box}\sum_{J \in P_{\nu^b}(I)}\int_{\Box}(|\E_{J}g|^{2} * \phi_{\cT_{J}, E/C_{0}})(|\E_{I'}g|^{10} * \phi_{\cT_{I'}, E/C_{0}}).
\end{align*}
This finishes the proof of the lemma.
\end{proof}

\begin{proof}[Proof of Lemma~\ref{lem:2a}]
	Suppose $b \leq a \leq 2b$. Let $\Box$ be a spatial cube of side length $\nu^{-2b}$.  Let $\gamma(\xi) = (\xi,\xi^2,\xi^3)$ and $\xi_1$, $\xi_2$ be the centers
of the intervals $I$ and $I'$. For $\alpha_1, \alpha_2, \alpha_3 \in \N$ with $\alpha_j \leq ja$ for $j = 1,2,3$, consider a parallelepiped
	$$
	\{x \in \R^3 \colon |x \cdot \gamma'(\xi_1)| \leq \nu^{-\alpha_1}, |x \cdot \gamma''(\xi_1)| \leq \nu^{-\alpha_2}, |x \cdot \gamma'''(\xi_1)| \leq \nu^{-\alpha_3}\}.
	$$
	Note that $|\E_I g|$ is morally locally constant on every translate of this parallelepiped. Tile $\R^3$ with essentially disjoint translates of this parallelepiped and let $\mc{T}_{\alpha_1,\alpha_2,\alpha_3}(I)$ be the parallelepipeds in this tiling. Similarly we tile $\R^3$ with essentially disjoint translates of the parallelepiped
	$$
	\{x \in \R^3 \colon |x \cdot \gamma'(\xi_2)| \leq \nu^{-\beta_1}, |x \cdot \gamma''(\xi_2)| \leq \nu^{-\beta_2}, |x \cdot \gamma'''(\xi_2)| \leq \nu^{-\beta_3}\}
	$$
	and define $\mc{T}_{\beta_1,\beta_2,\beta_3}(I')$ to be the parallelepipeds in this tiling whenever $\beta_1,\beta_2,\beta_3 \in \N$ with $\beta_j \leq jb$ for $j = 1,2,3$. Consider
	\beq\label{190505e3.9}
	\int_{\Box} \big( |\E_{I}g|^{2}*\phi_{\cT_I, E} \big) \big( |\E_{I'}g|^{10}* \phi_{\cT_{I'}, E}\big).
	\endeq
Since $b \leq a$, notice that there exists $c_T, c_{T'}$ for every $T\in \mc{T}_{a, 2b, 2b}(I)$ and every $T'\in \mc{T}_{b, 2b, 2b}(I')$ such that for all $x \in \R^3$,
\begin{align}\label{essentially_const}
\begin{aligned}
(|\E_{I}g|^{2}*\phi_{\cT_I, E})(x) &\sim_{E} \sum_{T\in \mc{T}_{a, 2b, 2b}(I)} c_T^2 \mathbbm{1}_{T}(x)\\
(|\E_{I'}g|^{10}*\phi_{\cT_{I'}, E})(x) &\sim_{E} \sum_{T'\in \mc{T}_{b, 2b, 2b}(I')} c_{T'}^{10} \mathbbm{1}_{T'}(x).
\end{aligned}
\end{align}
Since if $T \in \mc{T}_{a, 2b, 2b}(I)$ and $T' \in \mc{T}_{b, 2b, 2b}(I')$ then $T, T' \subset 2\Box$, it follows that \eqref{190505e3.9} can be bounded by
\beq\label{190505e3.11}
\int_{\Box} \big(\sum_{T\in \mc{T}_{a, 2b, 2b}(I); T\subset 2\Box} c_T^2 \mathbbm{1}_{T} \big) \big( \sum_{T'\in \mc{T}_{b, 2b, 2b}(I'); T'\subset 2\Box} c_{T'}^{10} \mathbbm{1}_{T'}\big)
\endeq
	For such $T$ and $T'$, we have a crucial geometric inequality
	\begin{equation} \label{eq:tube_intersect}
	\frac{|T \cap T'|}{|\Box|} \lesssim \nu^{-2} \frac{|T|}{|\Box|} \frac{|T'|}{|\Box|},
	\end{equation}
	because
	\[
	|T \cap T'|  \leq |\{ x \in \R^3 \colon |x \cdot \gamma'(\xi_1)| \lesssim \nu^{-a}, |x \cdot \gamma'(\xi_2)| \lesssim \nu^{-b},  |x \cdot \gamma''(\xi_2)| \lesssim \nu^{-2b} \}|
	\]
	the latter of which is comparable to
	\[
	\begin{split}
	 \nu^{-a} \nu^{-b} \nu^{-2b} \left| \det
	\left( \begin{array}{c}
	\gamma'(\xi_1) \\
	\gamma'(\xi_2) \\
	\gamma''(\xi_2)
	\end{array} \right)^{-1} \right| &\sim \frac{(\nu^{-a} \nu^{-2b} \nu^{-2b}) (\nu^{-b} \nu^{-2b} \nu^{-2b})}{(\nu^{-2b})^3} (\xi_1 - \xi_2)^{-2}\\
    & \lesssim \nu^{-2} \frac{|T| |T'|}{|\Box|}.
	\end{split}
	\]
	This implies
	\beq\label{190505e3.13}
	\begin{split}
	\eqref{190505e3.11}& \lesim \frac{\nu^{-2}}{|\Box|} \Big(\int_{\Box} \sum_{T\in \mc{T}_{a, 2b, 2b}(I); T\subset 2\Box} c_T^2 \mathbbm{1}_{T}\Big) \Big(\int_{\Box} \sum_{T'\in \mc{T}_{b, 2b, 2b}(I'); T'\subset 2\Box} c_{T'}^{10} \mathbbm{1}_{T'}\Big)\\
	& \lesim_{E} \frac{\nu^{-2}}{|\Box|} \big( \int_{\Box}  |\E_{I}g|^{2}*\phi_{\cT_I, E} \big) \big(\int_{\Box} |\E_{I'}g|^{10}* \phi_{\cT_{I'}, E}\big).
	\end{split}
	\endeq
By $L^2$ orthogonality and an argument that is essentially the same as that in \eqref{190514e3.8}, we have
\beq\label{190505e3.14}
\begin{split}
\int_{\Box}  |\E_{I}g|^{2}*\phi_{\cT_I, E}  \lesim_{E} \sum_{J\in P_{\nu^{2b}}(I)} \int_{\R^3} \big(|\E_{J}g|^{2}*\phi_{\cT_J, E/C_0} \big) w_{\Box, C_0 E}.
\end{split}
\endeq
By the definition of the weight $w_{\Box, C_0 E}$, the term \eqref{190505e3.13} can be bounded by
\begin{align*}
	\begin{split}
	\nu^{-2}\sum_{J\in P_{\nu^{2b}}(I)}  \sum_{\kappa\in \Z^3} \frac{(1+|\kappa|)^{-C_0 E}}{|\Box|} \big(\int_{\Box_{\nu^{-2b}\kappa}}|\E_{J}g|^{2}*\phi_{\cT_J, E/C_0} \big) \big(\int_{\Box} |\E_{I'}g|^{10}* \phi_{\cT_{I'}, E}\big)
	\end{split}
\end{align*}
Applying \eqref{190514e3.11}, with $b$ replaced by $2b$ shows that the above is bounded by
\begin{align*}
&\nu^{-2}\sum_{\kappa \in \Z^3}(1 + |\kappa|)^{-C_{0}E}|\kappa|^{E/C_0}\sum_{J \in P_{\nu^{2b}}(I)}(\inf_{x'\in \Box} |\E_{J}g|^{2}*\phi_{\cT_J, E/C_0}(x'))(\int_{\Box} |\E_{I'}g|^{10}* \phi_{\cT_{I'}, E})\\
&\lsm \nu^{-2}\sum_{J \in P_{\nu^{2b}}(I)}\int_{\Box}(|\E_{J}g|^{2}*\phi_{\cT_J, E/C_0})(|\E_{I'}g|^{10}* \phi_{\cT_{I'}, E}).
\end{align*}
Summing up over $\{\Box\}$, a partition of $\R^3$ into cubes of side length $\nu^{-2b}$, then finishes the proof of Lemma \ref{lem:2a}.
\end{proof}

\begin{proof}[Proof of Lemma~\ref{lem:3b}]
	Suppose $2b \leq a \leq 3b$. We may follow line by line the proof of Lemma~\ref{lem:2a}, except that
	\begin{itemize}
		\item the side length $\nu^{-2b}$ of the spatial cube $\Box$  replaced by $\nu^{-3b}$;
		\item $\mc{T}_{a,2b,2b}(I)$  replaced by $\mc{T}_{a,3b,3b}(I)$; and
		\item $\mc{T}_{b,2b,2b}(I')$  replaced by $\mc{T}_{b,2b,3b}(I')$.
	\end{itemize}
	This is because when $2b \leq a \leq 3b$, the uncertainty principle asserts that morally speaking, $|\E_I g|$ is locally constant on all tubes in $\mc{T}_{a,3b,3b}(I)$, and $|\E_{I'} g|$ is locally constant on all tubes in $\mc{T}_{b,2b,3b}(I')$. This shows that \eqref{essentially_const} holds with $\mc{T}_{a,2b,2b}(I)$ replaced by $\mc{T}_{a,3b,3b}(I)$, and $\mc{T}_{b,2b,2b}(I')$ replaced by $\mc{T}_{b,2b,3b}(I')$. The crucial geometric inequality \eqref{eq:tube_intersect} now follows since we still have
	\[
	|T \cap T'|  \leq |\{ x \in \R^3 \colon |x \cdot \gamma'(\xi_1)| \lesssim \nu^{-a}, |x \cdot \gamma'(\xi_2)| \lesssim \nu^{-b},  |x \cdot \gamma''(\xi_2)| \lesssim \nu^{-2b} \}|
	\]
	the latter of which is comparable to
	\[
	\begin{split}
	\nu^{-a} \nu^{-b} \nu^{-2b} \left| \det
	\left( \begin{array}{c}
	\gamma'(\xi_1) \\
	\gamma'(\xi_2) \\
	\gamma''(\xi_2)
	\end{array} \right)^{-1} \right|& \sim \frac{(\nu^{-a} \nu^{-3b} \nu^{-3b}) (\nu^{-b} \nu^{-2b} \nu^{-3b})}{(\nu^{-3b})^3} (\xi_1 - \xi_2)^{-2}\\
& \lesssim \nu^{-2} \frac{|T| |T'|}{|\Box|}.
	\end{split}
	\]
In lieu of \eqref{190505e3.14}, since now $\Box$ is a cube of side length $\nu^{-3b}$, and $|I| = \nu^a \geq \nu^{3b}$, we may apply $\ell^2 L^2$ decoupling, and bound instead
\begin{align*}
\begin{split}
\int_{\Box}  |\E_{I}g|^{2}*\phi_{\cT_I, E}  \lesim \sum_{J\in P_{\nu^{3b}}(I)} \int_{\R^3} \big(|\E_{J}g|^{2}*\phi_{\cT_J, E/C_0} \big) w_{\Box, C_0 E}.
\end{split}
\end{align*}
	This completes the proof of Lemma~\ref{lem:3b}.
\end{proof}

\section{The second bilinear constant $\mc{M}_{2, a, b}$} \label{sect:M2ab}
We will now prove the following result.
\begin{lemma} \label{lem:4.1}
	Let $1 \leq a \leq b$ and $\nu \in 2^{-2^{\mathbb{N}}} \cap (0,1/1000)$. Let $I$ be an interval of length $\nu^{a}$ and $I'$ be an interval of length $\nu^{b}$ such that $d(I, I')\ge 2\nu$. Then for every $\vep > 0$
there exists an absolute constant $C_0$ such that
	\beq\label{190305e4.1}
	\begin{split}
	& \int_{\R^3}\big( |\E_{I}g|^{4}*\phi_{\cT_I, E} \big) \big( |\E_{I'}g|^{8}* \phi_{\cT_{I'}, E}\big)\\
	& \lesim_{\varepsilon, E} \nu^{-(1+\varepsilon)(2b-2a) - C_0} \sum_{J\in P_{\nu^{2b-a}}(I)} \int_{\R^3}\big( |\E_{J}g|^{4}*\phi_{\cT_J, E/C_0} \big) \big( |\E_{I'}g|^{8}* \phi_{\cT_{I'}, E/C_0}\big).
	\end{split}
	\endeq
\end{lemma}

Once we prove this, by applying the definition of $\mc{M}_{2, a, b}$, we obtain that
\begin{align*}
\begin{split}
 \int_{\R^3}\big( |\E_{I}g|^{4}*\phi_{\cT_I, E} \big) &\big( |\E_{I'}g|^{8}* \phi_{\cT_{I'}, E}\big)\\
 & \lesim_{\varepsilon, E}\nu^{-(1+\varepsilon)(2b-2a) - C_0} \mc{M}_{2, 2b-a, b}(\delta, \nu, E/C_0)^{12}\times\\
 &\quad\bigg( \sum_{J \in P_{\delta}(I)} \|\E_J g\|_{L^{12}(\R^3)}^4 \bigg) \bigg( \sum_{J' \in P_{\delta}(I')} \|\E_{J'} g\|_{L^{12}(\R^3)}^4 \bigg)^{2}.
\end{split}
\end{align*}
This concludes the desired estimate in Lemma \ref{cor_small_ball_bi}. \\

The proof of Lemma \ref{lem:4.1} consists of two steps. In the first step, we will prove a decoupling inequality for the parabola at a ``small" spatial  scale. In the second step, we will combine this decoupling inequality with an (rigorous) interpretation of the uncertainty principle and a few changes of variables to finish the proof of Lemma \ref{lem:4.1}.

In addition to the weight functions defined in the notation section, we will also need to consider weight functions adapted to squares in $\R^2$ and intervals $I \subset \R$.
In particular, such weight functions will appear (and only appear) in the statement and proofs of Lemmas \ref{190209lemma9} and \ref{l4}.
To this end, given a square $B \subset \R^2$ centered at $c = (c_1, c_2)$ of side length $R$, define
$$\wt{w}_{B, E}(x) := (1 + \frac{|x_1 - c_1|}{R})^{-E}(1 + \frac{|x_2 - c_2|}{R})^{-E}$$
and
$$w_{B, E}(x) := (1 + \frac{|x - c|}{R})^{-E}$$
for $x = (x_1, x_2) \in \R^2$.
This is a slight abuse of notation from $w_{B, E}$ where $B$ is a cube in $\R^3$ but we hope the distinction will be clear from context.

Next, for an interval $I \subset \R$ centered at $c$ of length $R$,  we let
$$w_{I, E}(x) := (1 + \frac{|x - c|}{R})^{-E}$$
for $x \in \R$.

\subsection{$\ell^4 L^4$ decoupling on small spatial scales}\label{20190607sub4.1}

\begin{lemma}\label{190209lemma9}
Let
$$
(E_{I}g)(x) := \int_{I}g(\xi)e(\xi x_1 + \xi^{2}x_2)\, d\xi
$$
be the extension operator for the parabola associated to a dyadic interval $I\subset [0, 1]$.
Then for every $\varepsilon >0$, every $\delta \in 2^{-\N}$ and every square $B_{\delta^{-1}} \subset \R^2$ of side length $\delta^{-1}$, we have the decoupling inequality
\begin{align}\label{small_ball_decoupling_4}
\nms{E_{[0, 1]}g}_{L^{4}(w_{B_{\delta^{-1}}, E})} \lsm_{\varepsilon, E} \delta^{-\frac 1 4-\varepsilon}\Big(\sum_{J \in P_{\delta}([0, 1])}\nms{E_{J}g}_{L^{4}(w_{B_{\delta^{-1}}, E})}^{4}\Big)^{1/4}
\end{align}
for every function $g \colon [0, 1] \rightarrow \C$.
\end{lemma}

If $B_{\delta^{-1}}$ is replaced by $B_{\delta^{-2}}$ in \eqref{small_ball_decoupling_4}, then this estimate would be a trivial consequence of the $\ell^2$ decoupling inequality of Bourgain and Demeter \cite{BD15,BD17}. Also, if we were to prove \eqref{small_ball_decoupling_4} with the constant $\delta^{-1/4}$ replaced by $\delta^{-1/2}$, then this would follow easily by interpolation between $L^2$ and $L^{\infty}$; see Lemma~\ref{lem:trivialL4} below. It is worth mentioning that in \eqref{small_ball_decoupling_4}, the power of $\delta^{-1}$ is optimal. This can be seen by taking the function $g$ to be the indicator function of $[0, 1]$.

Lemma~\ref{190209lemma9} is also a special case of Demeter, Guth and Wang \cite[Theorem 3.1]{DGW19}. We provide a proof in our simpler special case.

\begin{proof}[Proof of Lemma \ref{190209lemma9}. ]
For readers familiar with the Bourgain-Guth iterations in \cite{BG11}, we first sketch a possible proof by making free use of the uncertainty principle and ignoring all Schwartz tails. This should make clear the ideas behind the rigourous proof, which will follow immediately after.

Let $\varepsilon > 0$ and $K \in 2^{\N}$ to be chosen depending only on $\varepsilon$. Let $\delta \in 2^{-\N}$ be such that $\delta^{-1} \geq K$. Then the Bourgain-Guth dichotomy classifies each square $B_K \subset B_{\delta^{-1}}$ of side length $K$ as either broad or narrow, and summing the resulting estimates gives
\[
\begin{split}
\|E_{[0, 1]} g\|^4_{L^4(B_{\delta^{-1}})} & \leq 10 \sum_{\alpha\in P_{K^{-1}}} \|E_{\alpha}g\|_{L^4(B_{\delta^{-1}})}^4 +C_K \sum_{\substack{\alpha_1, \alpha_2 \in P_{K^{-1}} \\ \dist(\alpha_1, \alpha_2)\ge 4/K}} \Big\| \prod_{j=1}^2 |E_{\alpha_j}g|^{1/2}\Big\|_{L^4(B_{\delta^{-1}})}^4.
\end{split}
\]
The second term on the right hand side can be bounded by
\[
\leq C_K \delta^2 \Big( \sum_{J\in P_{\delta}}\|E_{J} g\|^2_{L^2(B_{\delta^{-1}})}\Big)^2
\leq
C_K \delta^{-1} \sum_{J\in P_{\delta}}\|E_{J} g\|^4_{L^4(B_{\delta^{-1}})}
\]
where in the first estimate we used bilinear restriction followed by local $L^2$ orthogonality, and in the second estimate we used H\"{o}lder's inequality.
As a result, we obtain
\begin{equation} \label{eq:step1}
\|E_{[0, 1]} g\|^4_{L^4(B_{\delta^{-1}})} \leq 10 \sum_{\alpha\in P_{K^{-1}}} \|E_{\alpha}g\|_{L^4(B_{\delta^{-1}})}^4 + C_K \delta^{-1} \sum_{J\in P_{\delta}}\|E_{J} g\|^4_{L^4(B_{\delta^{-1}})}.
\end{equation}

We will rescale \eqref{eq:step1} as follows. Let $\sigma \in 2^{-\N}$ be such that $\sigma \geq (K \delta)^{1/2}$, and $I \subset [0,1]$ be a dyadic interval of length $\sigma$. We apply \eqref{eq:step1} to $g_{\sigma}$ in place of $g$ on squares of side length $\sigma^2 \delta^{-1}$, where $g_{\sigma}$ is the composition of $g$ with an affine map that maps $I$ bijectively onto $[0,1]$. Since $\|E_I g\|_{L^4(B_{\delta^{-1}})}^4 = \sigma^2 \|E_{[0,1]} g_{\sigma}\|_{L^4(R'_{\delta^{-1}})}^4$ where $R'_{\delta^{-1}}$ is a parallelogram of size $\sigma \delta^{-1} \times \sigma^2 \delta^{-1}$, which in turn can be covered by a union of $\sim \sigma^{-1}$ squares of side lengths $\sigma^2 \delta^{-1}$, we obtain
\[
\|E_I g\|_{L^4(B_{\delta^{-1}})}^4
\leq 10 \sum_{\alpha\in P_{\sigma/K}(I)} \|E_{\alpha} g\|^4_{L^4(B_{\delta^{-1}})}+ C_K \sigma^2 \delta^{-1} \sum_{J\in P_{\sigma^{-1} \delta}(I)} \|E_{J} g\|^4_{L^4(B_{\delta^{-1}})}.
\]
By interpolating a trivial bound at $L^{\infty}$ with the inequality at $L^2$ obtained via orthogonality, we can decouple the second term above from frequency scale $\sigma^{-1} \delta$ down to $\delta$, and obtain
\begin{equation} \label{eq:step2}
\|E_I g\|_{L^4(B_{\delta^{-1}})}^4
\leq 10 \sum_{\alpha\in P_{\sigma/K}(I)} \|E_{\alpha} g\|^4_{L^4(B_{\delta^{-1}})}+ C_K \delta^{-1} \sum_{J\in P_{\delta}(I)} \|E_{J} g\|^4_{L^4(B_{\delta^{-1}})}.
\end{equation}

We may now apply \eqref{eq:step2} repeatedly, for $\sigma = 1, K^{-1}, K^{-2}, \dots, K^{-(M-1)}$ where $M$ is the unique positive integer so that $K^{-(M-1)} \geq (K \delta)^{1/2} > K^{-M}$ (so roughly $K^{-M} \sim \delta^{1/2}$), and obtain
\[
\|E_{[0, 1]}g\|_{L^4(B_{\delta^{-1}})}^4 \le 10^M \sum_{\alpha\in P_{K^{-M}}} \|E_{\alpha}g\|_{L^4(B_{\delta^{-1}})}^4 + C_K M \delta^{-1} \sum_{J\in P_{\delta}} \|E_{J} g\|^4_{L^4(B_{\delta^{-1}})}.
\]
Finally, again by interpolating a trivial bound at $L^{\infty}$ with the inequality at $L^2$ obtained via orthogonality, the first term above can be estimated by
\[
C 10^M \Big( \frac{K^{-M}}{\delta} \Big)^{2} \sum_{\alpha\in P_{\delta}} \|E_{\alpha}g\|_{L^4(B_{\delta^{-1}})}^4.
\]
Since $K^{-M} / \delta \leq (K \delta)^{1/2} / \delta = K^{1/2} \delta^{-1/2}$, we then obtain
\[
\|E_{[0, 1]}g\|_{L^4(B_{\delta^{-1}})}^4 \le (C 10^M K \delta^{-1} + C_K M \delta^{-1}) \sum_{J\in P_{\delta}} \|E_{J} g\|^4_{L^4(B_{\delta^{-1}})}
\]
which via \cite[Lemma 4.1]{BD17} implies \eqref{small_ball_decoupling_4} because $M \leq \frac{1}{2} (1 + \frac{ \log \delta^{-1} }{ \log K })$ and $K$ can be chosen sufficiently large depending on $\varepsilon$.

Now that the idea of the proof is laid out, we will give a proof with more details, that proves a slightly more general statement and allows us to later deal with a general $C^3$ curve with curvature in place of the parabola. To state this slightly more general statement we need some notations. Suppose $\delta \in 2^{-\N}$ and $J \in P_{\delta}$. We denote by $T_J$ the parallelogram $\{(\xi,\eta) \in \R^2 \colon \xi \in J, \, |\eta - (a_J^2 + 2 a_J (\xi-a_J))| \leq \delta\}$; here $a_J$ is the left endpoint of $J$. Note that $\{T_J\}_{J \in P_{\delta}}$ is a family of essentially disjoint parallelograms of sizes $\sim \delta \times \delta$. We will denote by $f_J$ the inverse Fourier transform of $\mathbbm{1}_{T_J} \widehat{f}$, where $\mathbbm{1}_{T_J}$ is the indicator function of $T_J$. We will prove that for every $\varepsilon > 0$, if $\delta \in 2^{-\N}$ and $f = \sum_{J \in P_{\delta}} f_J$, then
\begin{equation} \label{eq:decoup_parabola_parallelogram}
\|f\|_{L^4(\R^2)} \lesssim_{\varepsilon} \delta^{-\frac{1}{4} - \varepsilon} \Big( \sum_{J \in P_{\delta}} \|f_{J}\|_{L^4(\R^2)}^4 \Big)^{1/4}.
\end{equation}

Indeed, let $K = K(\varepsilon)$ to be chosen depending only on $\varepsilon$. For each square $B_K \subset \R^2$ of side length $K$ and each $\alpha\in P_{K^{-1}}$, define
$$
c_{\alpha}(B_K):= \left( \frac{1}{|B_K|} \int_{B_K} |f_{\alpha}|^4 \right)^{1/4}.
$$
We will use the following form of uncertainty principle:
\begin{equation} \label{eq:uncertainty_2d}
c_{\alpha}(B_K) \leq C \inf_{x \in B_K} \left( \int_{\R^2} |f_{\alpha}(x-y)|^2 w_K(y) dy \right)^{1/2}
\end{equation}
where $w_K(y) := K^{-2} (1 + K^{-1} |y|)^{-100}$, which can be justified rigorously by noting that $f_{\alpha}$ is left unchanged by a Schwartz Fourier multiplier that is $1$ on a ball of radius $\sim K^{-1}$ containing the Fourier support of $f_{\alpha}$, and then applying the Cauchy-Schwarz inequality (note that $w_K(y) \sim w_K(y')$ whenever $|y-y'| \leq K$).

Now given such a square $B_K$, either there exists $\alpha^* \in P_{K^{-1}}$ such that $c_{\alpha}(B_K) \leq \frac{1}{K^{1/4}} c_{\alpha^*}(B_K)$ for all $\alpha \in P_{K^{-1}}$ with $\dist(\alpha,\alpha^*) > 4/K$, or there exists $\alpha^*, \alpha^{**} \in P_{K^{-1}}$, with $\dist(\alpha^*, \alpha^{**}) > 4/K$, so that $c_{\alpha}(B_K) \leq (K^{1/4} c_{\alpha^*}(B_K) c_{\alpha^{**}}(B_K))^{1/2}$. In the first case,
\[
\int_{B_K} |f|^4 \leq 10 \int_{B_K} |f_{\alpha^*}|^4 \leq 10 \sum_{\alpha \in P_{K^{-1}}} \int_{B_K} |f_{\alpha}|^4,
\]
while in the second case, the uncertainty principle \eqref{eq:uncertainty_2d} gives
\[
\begin{split}
\int_{B_K} |f|^4
\leq \, & C_K \int_{\R^2 \times \R^2} \left( \int_{B_K} |f_{\alpha^*}(x-y_1) f_{\alpha^{**}}(x-y_2)|^2 dx \right) w_K(y) dy \\
\leq \,  & C_K \sum_{\substack{\alpha_1, \alpha_2 \in P_{K^{-1}} \\ \dist(\alpha_1,\alpha_2) \geq 4/K}} \int_{\R^2 \times \R^2} \left( \int_{B_K} |f_{\alpha_1}(x-y_1) f_{\alpha_2}(x-y_2)|^2 dx \right) w_K(y) dy
\end{split}
\]
where we have written $w_K(y)$ as a shorthand for $ w_K(y_1) w_K(y_2)$.
Now let $B \subset \R^2$ be a square of side length $\delta^{-1}$. Then summing the previous estimates over all squares $B_K \subset B$ with side lengths $K$, we obtain
\begin{equation} \label{eq:local_L4}
\begin{split}
 & \int_{B} |f|^4
\leq 10 \sum_{\alpha \in P_{K^{-1}}} \int_{B} |f_{\alpha}|^4 \\
& \quad + C_K \sum_{\substack{\alpha_1, \alpha_2 \in P_{K^{-1}} \\ \dist(\alpha_1,\alpha_2) \geq 4/K}} \int_{\R^2 \times \R^2} \left( \int_{B} |f_{\alpha_1}(x-y_1) f_{\alpha_2}(x-y_2)|^2 dx \right) w_K(y) dy.
\end{split}
\end{equation}
To estimate the second term on the right, for each fixed $(y_1,y_2) \in \R^2 \times \R^2$, we apply bilinear restriction estimate in $\R^2$ to $F_1 := \eta_B(x) f_{\alpha_1}(x-y_1)$ and $F_2(x) := \eta_B(x) f_{\alpha_2}(x-y_2)$ where $\eta_B$ is a Schwartz function whose Fourier transform is supported in a ball of radius $\delta$ centered at $0$, and $|\eta| \geq 1$ on $B$:
\begin{lemma} \label{lem:bil_rest_plane}
If $0 < \delta \leq K^{-1} \leq 1$ and $S_1, S_2$ be $\delta$-neighborhoods of two arcs of the parabola $(\xi,\xi^2)$ that are of lengths $K^{-1}$ and at least $4/K$ apart, then for any $F_1, F_2 \colon \R^2 \to \C$ whose Fourier transforms are supported on $S_1$ and $S_2$ respectively, we have
\[
\int_{\R^2} |F_1 F_2|^2 \leq C_K \delta^2 \prod_{j=1}^2 \left( \int_{\R^2} |F_j|^2 \right).
\]
\end{lemma}
The second term on the right hand side of \eqref{eq:local_L4} is then bounded by
\[
 C_K \delta^2 \sum_{\substack{\alpha_1, \alpha_2 \in P_{K^{-1}} \\ \dist(\alpha_1,\alpha_2) \geq 4/K}} \int_{\R^2 \times \R^2} \prod_{j=1}^2 \left( \int_{\R^2} |\eta_B(x) f_{\alpha_j}(x-y_j)|^2 \right) w_K(y) dy.
\]
By local $L^2$ orthogonality,
\[
\int_{\R^2} |\eta_B(x) f_{\alpha_j}(x-y_j)|^2 \leq C \sum_{J \in P_{\delta}(\alpha_j)} \int_{\R^2} |\eta_B(x) f_J(x-y_j)|^2
\]
for $j = 1,2$. It follows that
\[
\int_{\R^2} |\eta_B(x) f_{\alpha_j}(x-y_j)|^2 \leq C \delta^{-3/2} \left( \sum_{J \in P_{\delta}} \int_{\R^2} |f_J(x-y_j)|^4 w_B(x) \right)^{1/2}
\]
where we applied Cauchy-Schwarz to both the sum in $J$ and the integral over~$x$. Integrating against $w_K(y_j)$, and using Cauchy-Schwarz again, we see that the second term on the right hand side of \eqref{eq:local_L4} is bounded by
\begin{equation} \label{eq:secondterm}
C_K \delta^{-1} \prod_{j=1}^2 \left( \int_{\R^2} \sum_{J \in P_{\delta}} \int_{\R^2} |f_J(x-y_j)|^4 w_B(x) w_K(y_j) dx dy_j \right)^{1/2}.
\end{equation}
Summing \eqref{eq:local_L4} over all squares $B \subset \R^2$ of side lengths $\delta^{-1}$, and applying Cauchy-Schwarz to the sum over $B$ of \eqref{eq:secondterm}, we obtain
\begin{equation} \label{eq:step1'}
 \int_{\R^2} |f|^4
\leq 10 \sum_{\alpha \in P_{K^{-1}}} \int_{\R^2} |f_{\alpha}|^4 + C_K \delta^{-1} \sum_{J \in P_{\delta}} \int_{\R^2} |f_J(x)|^4.
\end{equation}

We may now rescale \eqref{eq:step1'} and obtain, for every $\sigma \in 2^{-\N}$ with $\sigma \geq (K \delta)^{1/2}$ and every dyadic interval $I \subset [0,1]$ of length $\sigma$, that
\[
\int_{\R^2} |f_I|^4 \leq 10 \sum_{\alpha \in P_{\sigma/K}(I)} \int_{\R^2} |f_{\alpha}|^4 + C_K \sigma^2 \delta^{-1} \sum_{J \in P_{\sigma^{-1} \delta}(I)} \int_{\R^2} |f_J|^4.
\]
The second term on the right hand side can be bounded by the following lemma (with $N = \sigma^{-1}$), which is obtained by interpolation between $L^2$ orthogonality and a trivial bound at $L^{\infty}$:
\begin{lemma} \label{lem:trivialL4}
Let $\{F_j\}_{j=1}^N$ be a family of functions on $\R^2$ whose Fourier supports are contained in disjoint rectangles with sides parallel to coordinate axes. Then
\[
\Big\| \sum_{j=1}^N F_j \Big\|_{L^4(\R^2)} \leq C N^{\frac{1}{2}} \Big( \sum_{j=1}^N \|F_j\|_{L^4(\R^2)}^4 \Big)^{1/4}.
\]
\end{lemma}
We then get
\begin{equation} \label{eq:step2'}
\int_{\R^2} |f_I|^4 \leq 10 \sum_{\alpha \in P_{\sigma/K}(I)} \int_{\R^2} |f_{\alpha}|^4 + C_K \delta^{-1} \sum_{J \in P_{\delta}(I)} \int_{\R^2} |f_J|^4.
\end{equation}

We may now apply \eqref{eq:step2'} repeatedly, for $\sigma = 1, K^{-1}, K^{-2}, \dots, K^{-(M-1)}$ where $M$ is the unique positive integer so that $K^{-(M-1)} \geq (K\delta)^{1/2} > K^{-M}$, and obtain
\[
\int_{\R^2} |f|^4 \leq 10^M \sum_{\alpha \in P_{K^{-M}}} \int_{\R^2} |f_{\alpha}|^4 + C_K M \delta^{-1} \sum_{J \in P_{\delta}} \int_{\R^2} |f_J|^4.
\]
A final application of Lemma~\ref{lem:trivialL4} allows us to bound the first term on the right hand side above. Since $(K^{-M}/\delta)^{2} \leq ((K \delta)^{1/2})/\delta)^{2} = K \delta^{-1}$, we obtain
\[
\int_{\R^2} |f|^4 \leq  (C 10^M K \delta^{-1} + C_K M \delta^{-1})  \sum_{J \in P_{\delta}} \int_{\R^2} |f_J|^4.
\]
\eqref{eq:decoup_parabola_parallelogram} then follows because $M \leq \frac{1}{2}(1+\frac{\log \delta^{-1}}{\log K})$ and $K$ can be chosen sufficiently large depending on $\varepsilon$.

It is well-known that this implies \eqref{small_ball_decoupling_4}, because we can pick a Schwartz function $\eta$ whose Fourier transform is compactly supported in a ball of radius $\delta$ centered at the origin, and use \eqref{eq:decoup_parabola_parallelogram} to decouple $\eta E_{[0,1]}g = \sum_{J \in P_{\delta}} \eta E_J g$; each $\eta E_J g$ has Fourier support contained in a $\delta$-neighborhood of the parabola over $J$.
\end{proof}

For completeness, we include the short proofs of Lemma~\ref{lem:bil_rest_plane} and \ref{lem:trivialL4}.

\begin{proof}[Proof of Lemma~\ref{lem:bil_rest_plane}]
By Plancherel, it suffices to prove that
\[
\|\widehat{F_1}*\widehat{F_2}\|_{L^2(\R^2)} \leq C_K \delta \prod_{j=1}^2 \|\widehat{F_j}\|_{L^2(\R^2)}.
\]
Let $T$ be the bilinear operator given by
$$
T(G_1,G_2) := (\mathbbm{1}_{S_1} G_1)*(\mathbbm{1}_{S_2} G_2)
$$
where $\mathbbm{1}_{S_j}$ is the indicator function of $S_j$, for $j=1,2$. Then by Young's convolution inequality, $T$ is bounded from $L^1(\R^2) \times L^1(\R^2)$ to $L^1(\R^2)$ with norm $\leq 1$. Furthermore, since $\|\mathbbm{1}_{S_1}*\mathbbm{1}_{S_2}\|_{L^{\infty}(\R^2)} \leq C_K \delta^2$, we see that $T$ is bounded from $L^{\infty}(\R^2) \times L^{\infty}(\R^2)$ to $L^{\infty}(\R^2)$ with norm $\leq C_K \delta^2$. Thus by interpolation, $T$ is bounded from $L^2(\R^2) \times L^2(\R^2)$ to $L^2(\R^2)$ with norm $\leq C_K \delta$. Since $\widehat{F_1}*\widehat{F_2} = T(\widehat{F}_1,\widehat{F}_2)$, our claim follows.
\end{proof}

\begin{proof}[Proof of Lemma~\ref{lem:trivialL4}]
Let $R_1$, $\dots$, $R_N$ be disjoint rectangles with sides parallel to the axes containing the Fourier supports of $F_1$, $\dots$, $F_N$. Let $T$ be the $N$-linear operator so that $T(G_1,\dots,G_N)$ is the inverse Fourier transform of $\sum_{j=1}^N \mathbbm{1}_{R_j} \widehat{G_j}$.
Then by Plancherel, $T$ is bounded from $(L^2(\R^2))^N$ to $L^2(\R^2)$ with norm 1, and $T$ is bounded from $(L^{\infty}(\R^2))^N$ to a product $BMO(\R^2)$ with norm $\leq C N$. By interpolation between $L^2$ and $L^{\infty}$, we see that $T$ is bounded from $(L^4(\R^2))^N$ to $L^4(\R^2)$ with norm $\leq C N^{\frac{1}{2}}$. Since $\sum_{j=1}^N F_j = T(F_1,\dots,F_N)$, our claim follows.
\end{proof}

Finally, we state and prove the following generalization of Lemma~\ref{190209lemma9}:

\begin{lemma}\label{lem:small_ball_decoupling_4_general}
Let $\gamma \colon [0,1] \to \R^2$ be a $C^3$ curve with
\[
\|\gamma(\xi)\|_{C^3} \leq 100 \quad \text{and} \quad |\gamma'(\xi) \wedge \gamma''(\xi)| \geq \frac{1}{4}.
\]
Let
$$
(E_{I}g)(x) := \int_{I}g(\xi)e(\gamma(\xi) \cdot x)\, d\xi
$$
be the extension operator for $\gamma$ associated to a dyadic interval $I\subset [0, 1]$.
Then for every $\varepsilon >0$, every $\delta \in 2^{-\N}$ and every square $B_{\delta^{-1}} \subset \R^2$ of side length $\delta^{-1}$, we have the decoupling inequality
\begin{align}\label{small_ball_decoupling_4_general}
\nms{E_{[0, 1]}g}_{L^{4}(w_{B_{\delta^{-1}}, E})} \lsm_{\varepsilon, E} \delta^{-\frac 1 4-\varepsilon}\Big(\sum_{J \in P_{\delta}([0, 1])}\nms{E_{J}g}_{L^{4}(w_{B_{\delta^{-1}}, E})}^{4}\Big)^{1/4}
\end{align}
for every function $g \colon [0, 1] \rightarrow \C$.
\end{lemma}

\begin{proof}
This follows from \eqref{eq:decoup_parabola_parallelogram} via an iteration that goes back to Pramanik and Seeger \cite{PS07}. The key is that on any interval $I \subset [0,1]$, we may Taylor expand $\gamma$ around the left endpoint $a_I \in I$ and obtain
\[
\gamma(\xi) = \gamma(a_I) + \gamma'(a_I) (\xi-a_I) +  \frac{1}{2} \gamma''(\xi_0) (\xi-a_I)^2 + O(|I|^3).
\]
An affine transformation on $\R^2$ will transform the curve parametrized by the first three terms of the above Taylor expansion to the parabola over $[0,|I|]$. The inverse of this affine transformation is given by $A_I(\xi,\eta) := \gamma(a_I) + L_I(\xi,\eta)$ where
\[
L_I(\xi,\eta):= \xi \gamma'(a_I) + \eta \frac{\gamma''(a_I)}{2}
\]
is a linear map on $\R^2$ with determinant $\geq 1/8$. Hence the linear part of $A_I^{-1}$ has norm bounded above independent of $I$, and the curve $\gamma(I)$ is transformed under $A_I^{-1}$ to a curve whose distance from the parabola over $[0,|I|]$ is $O(|I|^3)$. So if $|I| \sim \delta^{1/3}$ and $f_I$ is a function on $\R^2$ whose Fourier transform is supported in an $\delta$-neighborhood of $\gamma(I)$, then
$S_I f_I$
has Fourier transform supported in a $O(\delta)$-neighborhood of the parabola over $[0,|I|]$ where
\begin{equation} \label{eq:fIrescaled}
S_I f(x) := e^{-2\pi i  \gamma(a_I) \cdot L_I^{-t} x } f(L_I^{-t} x)
\end{equation}
is defined so that $\widehat{S_I f}(\xi,\eta) = (\det L_I) \widehat{f}(A_I (\xi,\eta))$. We may then use \eqref{eq:decoup_parabola_parallelogram} to decouple $f_I$ in $L^4(\R^2)$ from frequency scale $\delta^{1/3}$ down to frequency scale $\delta$.

As a result, to decouple down to frequency scale $\delta$, it suffices to decouple down to frequency scale $\delta^{1/3}$. But then we may repeat this argument, and reduce ourselves to decoupling down to frequency scale $\delta^{1/9}$, $\delta^{1/27}$, $\dots$. Hence it suffices to decouple from frequency scale $1$ down to $\delta^{1/3^a}$ for some large positive integer $a$. But that can be done by a trivial decoupling, incurring only a $\delta^{-\varepsilon}$ loss if $a$ is sufficiently big.

To formalize these ideas, let $\delta \in 2^{-\N}$. For each $J \in P_{\delta}$, let $f_J$ be a function on $\R^2$ whose Fourier support is contained in a $\delta$-neighborhood of $\gamma(J)$. If $\sigma \in 2^{-\N \cup \{0\}}$ with $\sigma > \delta$, and $I \in P_{\sigma}$, we write $f_I := \sum_{J \in P_{\delta}, \, J \subset I} f_J$. Then the Fourier support of $f_I$ is contained in the $\sigma$-neighborhood of $\gamma(I)$. We will prove that for every $\varepsilon > 0$,
\begin{equation} \label{eq:C3curve_parallelogram}
\| f_{[0,1]} \|_{L^4(\R^2)} \lesssim_{\varepsilon} \delta^{-\frac{1}{4}-\varepsilon} \Big( \sum_{J \in P_{\delta}} \|f_J\|_{L^4(\R^2)}^4 \Big)^{1/4}.
\end{equation}
First, we pick positive integer $a$ so that $3^{-a} < \varepsilon/2$, and write $\delta = 2^{-q 3^a -r}$ for some non-negative integers $q$ and $r$ with $r < 3^a$. Then we trivially decouple down to frequency scale $2^{-q}$:
\[
\| f_{[0,1]} \|_{L^4(\R^2)}
\leq 2^q \Big( \sum_{I_0 \in P_{2^{-q}}} \| f_{I_0} \|_{L^4(\R^2)}^4 \Big)^{1/4}.
\]
We are now at frequency scale $\delta_0 := 2^{-q}$ and will successively decouple down to frequency scales $\delta_i := 2^{-q 3^i}$ for $i = 1, 2, \dots, a$. Indeed, for $i = 1, 2, \dots, a$ and $I_{i-1} \in P_{\delta_{i-1}}$, if $I_i \in P_{\delta_i}(I_{i-1})$, then the function $S_{I_{i-1}} f_{I_{i-1}}$ defined by \eqref{eq:fIrescaled} has Fourier transform supported in an $O(\delta_i)$-neighborhood of the parabola over $[0,\delta_{i-1}]$. So we may apply \eqref{eq:decoup_parabola_parallelogram} to decouple $S_{I_{i-1}} f_{I_{i-1}} = \sum_{I_i \in P_{\delta_i}(I_{i-1})} S_{I_{i-1}} f_{I_i}$ and obtain
\[
\|f_{I_{i-1}}\|_{L^4(\R^2)} \lesssim_{\varepsilon} \left( \frac{\delta_i}{\delta_{i-1}} \right)^{-\frac{1}{4}-\frac{\varepsilon}{2}} \Big( \sum_{I_i \in P_{\delta_i}(I_{i-1})} \| f_{I_i} \|_{L^4(\R^2)}^4 \Big)^{1/4}.
\]
Hence
\[
\| f_{[0,1]} \|_{L^4(\R^2)}
\lesssim_{\varepsilon} 2^q \delta_a^{-\frac{1}{4}-\frac{\varepsilon}{2}} \Big( \sum_{I_a \in P_{\delta_a}} \| f_{I_a} \|_{L^4(\R^2)}^4 \Big)^{1/4}.
\]
Note that $2^q \leq \delta^{-\frac{1}{3^a}} \leq \delta^{-\frac{\varepsilon}{2}}$ and $\delta_a^{-1} = 2^{q 3^a} \leq \delta^{-1}$. Finally we trivially decouple from frequency scale $\delta_a = 2^{-q 3^a}$ to frequency scale $\delta = 2^{-q 3^a - r}$: since $2^r \leq 2^{3^a} = O_{\varepsilon}(1)$, we obtain \eqref{eq:C3curve_parallelogram}.
This implies \eqref{small_ball_decoupling_4_general}, in the same way that \eqref{eq:decoup_parabola_parallelogram} implies \eqref{small_ball_decoupling_4}.
\end{proof}

\begin{rem}\label{remark_2_20190607}
The use of $\ell^4$ sum on the right hand side of \eqref{small_ball_decoupling_4} determines that the current new argument that is used to prove Theorem \ref{main}, which is inspired by \cite{Woo16} and \cite{HB15}, cannot be used to recover \eqref{decoupling_l_2}.
\end{rem}

\subsection{The proof of Lemma \ref{lem:4.1}}
We can assume that $a < b$ since when $a = b$, there is nothing to show.
By affine invariance, we may assume that $I' = [0, \nu^b]$. Notice that $|I'|=\nu^{b}$, therefore the function $|\E_{I'} g|^8$
is essentially constant on every axis-parallel slab of dimension $\nu^{-b}\times \nu^{-2b}\times \nu^{-3b}$. Here the short side of
length $\nu^{-b}$ is along the $x_1$-axis and the side of medium length is along the $x_2$-axis.

Since $I' = [0, \nu^b]$, there are absolute constants $c < 1$ and $C > 1$ such that $\cT_{I'}$ contains the axis parallel rectangular
box of dimension $c\nu^{-b} \times c\nu^{-2b} \times c\nu^{-3b}$ centered at the origin and is contained in the axis parallel rectangular
box of dimension $C\nu^{-b} \times C\nu^{-2b} \times C\nu^{-3b}$ centered at the origin.

Let $\Box$ denote an arbitrary axis-parallel rectangular box of dimension $\nu^{-b} \times \nu^{-2b} \times \nu^{-3b}$. To estimate left hand side of \eqref{190305e4.1}, we first consider
\beq\label{190305e4.17}
\int_{\Box} \big( |\E_{I}g|^{4}*\phi_{\cT_I, E} \big) \big( |\E_{I'}g|^{8}* \phi_{\cT_{I'}, E}\big)
\endeq
Notice that for every $x, x'\in \Box$, we have
\begin{align*}
|\E_{I'}g|^{8}* \phi_{\cT_{I'}, E}(x) \sim_E |\E_{I'}g|^{8}* \phi_{\cT_{I'}, E}(x').
\end{align*}
Therefore, we bound \eqref{190305e4.17} by
\begin{align}\label{190505e4.20}
\Big( \sup_{x\in \Box}|\E_{I'}g|^{8}* \phi_{\cT_{I'}, E}(x)\Big) \int_{\Box} \big( |\E_{I}g|^{4}*\phi_{\cT_I, E} \big).
\end{align}
We keep the former factor as is for a while and focus on the latter factor. We first write it as
\beq\label{190505e4.21}
\int_{\R^3} \big[\int_{\Box_y} |\E_I g(x) |^4dx\big] \phi_{\cT_{I, E}}(y) dy,
\endeq
where $\Box_y:=\Box-y$. We will prove the following.

\begin{lemma}\label{l4}
Let $\nu \in 2^{-2^{\mathbb{N}}} \cap (0,1/1000)$, $I = [d, d + \nu^a]$ with $|d| \geq 2\nu$, and $\Delta$ a square in the $(x_2, x_3)$-plane of side length $\nu^{-2b}$. For every fixed $x_1\in \R$, there exists an absolute constant $C_0$ such that
\begin{align*}
 \int_{\R^2}&|(\E_{I}g)(x)|^{4}w_{\Delta, E}(x_2, x_3)\, dx_2\, dx_3\\
& \lsm_{\vep, E} \nu^{-(1 + \vep)(2b - 2a) - C_0}\sum_{J \in P_{\nu^{2b - a}}(I)}\int_{\R^2}|(\E_{J}g)(x)|^{4}w_{\Delta, E}(x_2, x_3)\, dx_2 \, dx_3
\end{align*}
for every $\vep > 0$.
\end{lemma}
First let's see how to use Lemma \ref{l4} to finish the proof.
By applying Lemma \ref{l4} (with $E$ replaced by $100 E$) and Fubini, we can bound \eqref{190505e4.21} by
\begin{align*}
\begin{split}
& \lsm_{\vep, E} \nu^{-(1 + \vep)(2b - 2a) - C_0}\sum_{J \in P_{\nu^{2b - a}}(I)}\int_{\R^3} \Big[\int_{\R^3}|(\E_{J}g)(x-y)|^{4}w_{\Box, E}(x)dx\Big] \phi_{\cT_{I, E}}(y) dy\\
& \lesim_{\vep, E} \nu^{-(1 + \vep)(2b - 2a) - C_0}\sum_{J \in P_{\nu^{2b - a}}(I)} \int_{\R^3} \big[|\E_J g|^4* \phi_{\cT_{I, E}}\big] w_{\Box, E}.
\end{split}
\end{align*}

In light of Lemma~\ref{eq39}, we have obtained that
\begin{align*}
\begin{split}
 \eqref{190505e4.20} \lesim_{\vep, E} &\nu^{-(1 + \vep)(2b - 2a) - C_0}\times\\
&\Big( \sup_{x\in \Box}|\E_{I'}g|^{8}* \phi_{\cT_{I'}, E}(x)\Big)\sum_{J \in P_{\nu^{2b - a}}(I)} \int_{\R^3} \big[|\E_J g|^4* \phi_{\cT_{J, E/C_0}}\big] w_{\Box, E}.
\end{split}
\end{align*}
It remains to prove that
\begin{align*}
\begin{split}
 \sum_{\Box} \Big( \sup_{x\in \Box}|\E_{I'}g|^{8}* \phi_{\cT_{I'}, E}(x)\Big)& \int_{\R^3} \big[|\E_J g|^4* \phi_{\cT_{J, E/C_0}}\big] w_{\Box,  E} \\
& \lesim_E \int_{\R^3}\big( |\E_{J}g|^{4}*\phi_{\cT_J, E/C_0} \big) \big( |\E_{I'}g|^{8}* \phi_{\cT_{I'}, E/C_0}\big).
\end{split}
\end{align*}
But the proof of this is essentially the same as that in the steps \eqref{190514e3.10}, \eqref{190514e3.11} and \eqref{190514e3.11a}:
one would bound the left hand side by
\[
\begin{split}
\sum_{\Box} \sum_{\kappa \in \Z^3} (1+|\kappa|)^{-E}  \Big( \sup_{x\in \Box}|\E_{I'}g|^{8}* \phi_{\cT_{I'}, E}(x)\Big)  \int_{\Box_{\nu^{-b} \circ \kappa}} \big[|\E_J g|^4* \phi_{\cT_{J, E/C_0}}\big]
\end{split}
\]
where $\nu^{-b} \circ \kappa := (\nu^{-b} \kappa_1, \nu^{-2b} \kappa_2, \nu^{-3b} \kappa_3)$ for $\kappa = (\kappa_1, \kappa_2, \kappa_3)$,
and use the following inequalities: we use
$$
\sup_{x\in \Box_{\nu^{-b} \circ \kappa}} |\E_{J}g|^{4}*\phi_{\cT_J, E/C_0}(x)\lesssim_{E} |\kappa|^{E/C_0} \inf_{x'\in \Box} |\E_{J}g|^{4}*\phi_{\cT_J, E/C_0}(x'),
$$
(here we used that the side lengths of $\cT_J$ are longer than those of $\Box$, which holds because $b \geq a$), and that
$$
\sup_{x \in \Box}|\E_{I'}g|^{8} \ast \phi_{\cT_{I'}, E}(x) \sim_{E} \inf_{x \in \Box}|\E_{I'}g|^{8} \ast \phi_{\cT_{I'}, E}(x).
$$

\subsection{The proof of Lemma \ref{l4}}
In the proof, to avoid using too many subscripts, we will use $(x, y, z)$ to stand for a point in $\R^3$ rather than $(x_1, x_2, x_3)$.

\begin{proof}[Proof of Lemma \ref{l4}]
By the same reasoning as in Lemma 2.5 and Proposition 2.6 of \cite{Li17}, it suffices to prove instead
\begin{align}\label{l4main}
\begin{aligned}
& \int_{\R^2}|(\E_{I}g)(x, y, z)|^{4}\mathbbm{1}_{\Delta}(y, z)\, dy\, dz\\
& \lsm_{\vep, E} \nu^{-(1 + \vep)(2b - 2a) - C_0}\sum_{J \in P_{\nu^{2b - a}}(I)}\int_{\R^2}|(\E_{J}g)(x, y, z)|^{4}\wt{w}_{\Delta, 10E}(y, z)\, dy\, dz
\end{aligned}
\end{align}
and furthermore, by shifting $y$ and $z$, it suffices to show this only in the case when $\Delta$ is centered at the origin.
Let $\ov{\Delta}$ be the square centered at the origin with coordinates $(\nu^{-2b}, 0)$, $(-\nu^{-2b}, 0)$, $(0, \nu^{-2b})$, and
$(0, -\nu^{-2b})$. Since $\Delta \subset \ov{\Delta}$, it suffices to show \eqref{l4main} with $\mathbbm{1}_{\Delta}(y, z)$ on the left hand
side replaced with $\mathbbm{1}_{\ov{\Delta}}(y, z)$. This small reduction will make the algebra later simpler.

Expanding the left hand side gives
\begin{align*}
\int_{\R^2}|\int_{d}^{d + \nu^a}g(t)e(tx)e(t^{2}y + t^{3}z)\, dt|^{4}\mathbbm{1}_{\ov{\Delta}}(y, z)\, dy\, dz.
\end{align*}
Rescaling $[d, d + \nu^a]$ to $[0, 1]$ shows that the above is equal to
\begin{equation}\label{e4.24}
\begin{aligned}
\nu^{4a}\int_{\R^2}|\int_{0}^{1}&g(d + \nu^{a}t)e((d + \nu^a t)x)\times\\
&e(\nu^{a}t(2dy + 3d^{2}z) + \nu^{2a}t^{2}(y + (3d + \nu^{a}t)z))\, dt|^{4}\mathbbm{1}_{\ov{\Delta}}(y, z)\, dy\, dz.
\end{aligned}
\end{equation}
Before we proceed, let us first describe the idea. It will become clear why we organize different terms in the phase function
as above. Notice that $\nu^{a}$ is an extremely small number. We will treat $3d + \nu^{a}t$ as a small perturbation of $3d$ and end up looking at the extension
operator for a (perturbed) parabola.

To make this idea precise, we make the change of variables
\begin{align*}
\begin{pmatrix}
y'\\ z'
\end{pmatrix} =
\begin{pmatrix}
2\nu^{a}d & 3\nu^{a}d^{2}\\\nu^{2a} & 3\nu^{2a}d
\end{pmatrix}
\begin{pmatrix}
y\\z
\end{pmatrix}.
\end{align*}
Denote the matrix above by $T$ and let $G(t, x) := g(d + \nu^{a}t)e((d + \nu^a t)x)$. Then using that $|d| \geq 2\nu$, \eqref{e4.24} is bounded by
\begin{align}\label{l4eq1}
\nu^{a - 2}\int_{\R^2}|\int_{0}^{1}G(t, x)e(y'(t - \frac{\nu^{2a}}{3d^{2}}t^3) + z'(t^{2} + \frac{2\nu^{a}}{3d}t^3))\, dt|^{4}\mathbbm{1}_{T(\ov{\Delta})}(y', z')\, dy'\, dz'.
\end{align}
Ignoring the weight for the moment, we will now want a decoupling theorem for the curve $\gamma(t) = (t - \frac{\nu^{2a}}{3d^2}t^3, t^{2} + \frac{2\nu^{a}}{3d}t^3)$ which is a small
perturbation of the parabola. But this comes from Lemma~\ref{lem:small_ball_decoupling_4_general}: indeed,
\[
|\gamma'(t) \wedge \gamma''(t)| = \left| \det \left(
\begin{array}{cc}
1 - \frac{\nu^{2a}}{d^2} t^2 & -2 \frac{\nu^{2a}}{d^2} t  \\
2t + \frac{2\nu^a}{d}t^2 & 2 + \frac{4 \nu^a}{d} t
\end{array}
\right) \right|
= \left(1 + \frac{\nu^a}{d} t \right)^2
\]
which is obviously $\geq 1$ for $t \in [0,1]$ if $d \geq 2\nu$; if on the other hand $d \leq -2\nu$, then for $t \in [0,1]$, $1 + \frac{\nu^a}{d} t \geq 1 - \frac{\nu^a}{2\nu} \geq \frac{1}{2}$ so $|\gamma'(t) \wedge \gamma''(t)| \geq \frac{1}{4}$ and Lemma~\ref{lem:small_ball_decoupling_4_general} will apply.

Note that $T(\ov{\Delta})$ is the parallelogram centered at the origin with vertices at the points
\begin{align*}
A: &(2d\nu^{-2b + a}, \nu^{-2b + 2a}),\\
B: &(-2d\nu^{-2b + a}, -\nu^{-2b + 2a}),\\
C: &(3d^{2}\nu^{-2b + a}, 3d\nu^{-2b + 2a}),\\
D: &(-3d^{2}\nu^{-2b + a}, -3d\nu^{-2b + 2a}).
\end{align*}
We have two cases: either $2\nu \leq |d| \leq 1/1000$ or $1/1000 < |d| \leq 1$. We will only focus on the former case. The latter case is slightly easier, as we have $O(1)$ separation.
We split the former case into two further cases $d > 0$ and $d < 0$. Again we only focus on the former case $d>0$. The proof for the other case is similar. \\

We will want to cover $T(\ov{\Delta})$ (a rotated thin parallelogram) by squares roughly of
side length $\nu^{-2b + 2a - 1}$. To simplify working with the weight functions adapted
to each of these squares we rotate this parallelogram so that the longest diagonal is on the
$y'$-axis and then we cover this rotated parallelogram with axis-parallel squares.

Since $d$ is sufficiently small, the longest diagonal is created by connecting the points $A$ and $B$ which lies on the line
$z' = \frac{\nu^a}{2d}y'$. Let $\ta$ be such that $\tan\ta = \frac{\nu^{a}}{2d}$ and let $R_{\ta}$ be the rotation matrix
that rotates by an angle $\ta$ in the counterclockwise direction. Therefore $R_{\ta}^{-1}T(\ov{\Delta})$ is a parallelogram
with the line connecting $R_{\ta}^{-1}A$ and $R_{\ta}^{-1}B$ on the $y'$-axis.
The $y'$-coordinate of $R_{\ta}^{-1}A$ is $$\nu^{-2b + a}(2d\cos\ta + \nu^{a}\sin\ta).$$
and the $z'$-coordinate of $R_{\ta}^{-1}C$ is
$$3d\nu^{-2b + a}(-d\sin\ta + \nu^{a}\cos\ta).$$

We can find an integer $N$ such that $N \mid \nu^{-2b + 2a}$ and $\frac{1}{2N} \leq d \leq \frac{1}{N}$. Indeed write $\nu = 2^{-2^{\alpha}}$
for some $\alpha \in \N$ sufficiently large. Since we want $N \leq \frac{1}{d} \leq 2N$ and know $1000 \leq \frac{1}{d} \leq \frac{1}{2}\nu^{-1}$,
choose $N$ from the set $\{2^9, 2^{10}, \ldots, 2^{2^{\alpha} - 2}\}$. Then $N \mid \nu^{-1}$ and since $2b - 2a \geq 2$, $N \mid \nu^{-2b + 2a}$.

Therefore $R_{\ta}^{-1}T(\ov{\Delta})$ is contained in a rectangle $\Delta'$ centered at the origin of length
\begin{align*}
2\nu^{-2b + a}(2d\cos\ta + \nu^{a}\sin\ta) = 2\nu^{-2b + a}\sqrt{4d^{2} + \nu^{2a}} \leq 6d\nu^{-2b + a} \leq \frac{6}{N}\nu^{-2b + a}
\end{align*}
and height
\begin{align*}
6d\nu^{-2b + a}(-d\sin\ta + \nu^{a}\cos\ta) = 3d\nu^{-2b + 2a}\frac{2d}{\sqrt{4d^{2} + \nu^{2a}}} \leq 6d\nu^{-2b + 2a} \leq \frac{6}{N}\nu^{-2b + 2a}.
\end{align*}
Partition this rectangle into $\nu^{-a}$ squares $\{\Box\}$ of side length $\frac{1}{N}\nu^{-2b + 2a}$.
Thus in this case we have shown that
\begin{align*}
\mathbbm{1}_{T(\ov{\Delta})}(y', z') &\leq \sum_{\Box}\mathbbm{1}_{\Box}(R_{\ta}^{-1}(y', z'))\\
 &\lsm_{E} \sum_{\Box}w_{\Box, 100E}(R_{\ta}^{-1}(y', z'))= \sum_{\Box}w_{B(0, \frac{1}{N}\nu^{-2b + 2a}), 100E}((y', z') - R_{\ta}c_{\Box}).
\end{align*}
where the last equality we have used that $w_{B(0, R), 100E}(x)$ is a radial function.
Therefore \eqref{l4eq1} is
\begin{align*}
\lsm_{E} \nu^{a - 2}\sum_{\Box}\int_{\R^2}|\int_{0}^{1}G(t, x)e(y'(t - \frac{\nu^{2a}}{3d^2}t^3) + &z'(t^2 + \frac{2\nu^a}{3d}t^3))\, dt|^{4}\times\\
&w_{B(R_{\ta}c_{\Box}, \frac{1}{N}\nu^{-2b + 2a}), 100E}(y', z')\, dy'\, dz'.
\end{align*}
Since $N \mid \nu^{-2b + 2a}$ and $N\nu^{2b - 2a} \in 2^{-\N}$, applying Lemma \ref{lem:small_ball_decoupling_4_general} shows that we can decouple to frequency scale $N\nu^{2b - 2a}$.
That is, the above is
\begin{align*}
\lsm_{\vep, E} &\frac{1}{N^{1 + \vep}}\nu^{-(2b - 2a)(1 + \vep) - 2}\times\\
&\sum_{\Box}\sum_{J \in P_{N\nu^{2b - 2a}}}\nu^{a}\int_{\R^2}|\int_{J}G(t, x)e(y'(t - \frac{\nu^{2a}}{3d^2}t^3) + z'(t^2 + \frac{2\nu^{a}}{3d}t^3))\, dt|^{4}\times\\
&\hspace{2.4in}w_{B(R_{\ta}c_{\Box}, \frac{1}{N}\nu^{-2b + 2a}), 100E}(y', z')\, dy'\, dz'.
\end{align*}
By undoing the change of variables, one controls
the above by
\begin{align*}
\lsm_{\vep, E}&\frac{d^2}{N^{1 + \vep}}\nu^{-(2b - 2a)(1 + \vep) - 2}\times\\
&\sum_{J \in P_{N\nu^{2b - a}}(I)}\int_{\R^2}|(\E_{J}g)(x, y, z)|^{4}(\sum_{\Box}w_{B(R_{\ta}c_{\Box}, \frac{1}{N}\nu^{-2b + 2a}), 100E}(T(y, z)))\, dy\, dz.
\end{align*}
Since $N \sim d^{-1}$, we can use the triangle inequality to decouple $I$ from frequency
scale $N\nu^{2b - a}$ to scale $\nu^{2b - a}$, losing only a factor of $O(N^3)$. Therefore the above is
\begin{align*}
\lsm_{\vep, E}&\nu^{-(2b - 2a)(1 + \vep) - 2}\times\\
&\sum_{J \in P_{\nu^{2b - a}}(I)}\int_{\R^2}|(\E_{J}g)(x, y, z)|^{4}(\sum_{\Box}w_{B(R_{\ta}c_{\Box}, \frac{1}{N}\nu^{-2b + 2a}), 100E}(T(y, z)))\, dy\, dz.
\end{align*}
Thus we will have proved \eqref{l4main} in the case when $2\nu \leq d \leq 1/1000$ provided we can show that
\begin{align}\label{weightmain}
\sum_{\Box}w_{B(R_{\ta}c_{\Box}, \frac{1}{N}\nu^{-2b + 2a}), 100E}(T(y, z)) \lsm_{E} \wt{w}_{\Delta, 10E}(y, z).
\end{align}
We have
\begin{align*}
\sum_{\Box}w_{B(R_{\ta}c_{\Box}, \frac{1}{N}\nu^{-2b + 2a}), 100E}(T(y, z)) = \sum_{\Box}w_{\Box, 100E}(R_{\ta}^{-1}T(y, z)).
\end{align*}
Since $(1 + |y|)(1 + |z|) \leq (1 + |(y, z)|)^2$, to show \eqref{weightmain}, it suffices to show that
\begin{align}\label{weightmain2}
\sum_{\Box}\wt{w}_{\Box, 50E}(R_{\ta}^{-1}T(y, z)) \lsm_{E} \wt{w}_{\Delta, 10E}(y, z).
\end{align}
Writing the centers of the $\Box$ that partition $\Delta'$ as $(c_{\Box, 1}, c_{\Box, 2})$, we have
\begin{align*}
\sum_{\Box}\mathbbm{1}_{I(c_{\Box, 1}, \frac{1}{N}\nu^{-2b + 2a})}(y)\mathbbm{1}_{I(c_{\Box, 2}, \frac{1}{N}\nu^{-2b + 2a})}(z) = \mathbbm{1}_{I(0, \frac{6}{N}\nu^{-2b + a})}(y)\mathbbm{1}_{I(0, \frac{6}{N}\nu^{-2b + 2a})}(z)
\end{align*}
where $I(a, L)$ is the interval $[a - L/2, a + L/2]$.
By the proof of Lemma 2.1 and Remark 2.2 of \cite{Li17},
\begin{align*}
(\mathbbm{1}_{I(c_{\Box, 1}, \frac{1}{N}\nu^{-2b + 2a})} \ast w_{I(0, \frac{1}{N}\nu^{-2b + 2a}), 50E})(y) \gtrsim_{E} (\frac{1}{N}\nu^{-2b + 2a})w_{I(c_{\Box, 1}, \frac{1}{N}\nu^{-2b + 2a}), 50E}(y)
\end{align*}
and similarly for the $z$-coordinate, the left hand side of \eqref{weightmain2} is
\begin{align*}
\lsm_{E} &(\frac{1}{N}\nu^{-2b + 2a})^{-2}\times\\
&(\mathbbm{1}_{I(0, \frac{6}{N}\nu^{-2b + a})}^{y}\mathbbm{1}_{I(0, \frac{6}{N}\nu^{-2b + 2a})}^{z} \ast w_{I(0, \frac{1}{N}\nu^{-2b + 2a}), 50E}^{y}w_{I(0, \frac{1}{N}\nu^{-2b + 2a}), 50E}^{z})(R_{\ta}^{-1}T(y, z))
\end{align*}
where here we have used $\mathbbm{1}_{I(0,R)}^{y}$ to be shorthand for $\mathbbm{1}_{I(0, R)}(y)$ and similarly for $\mathbbm{1}_{I(0, R)}^{z}$, $w_{I(0, R)}^{y}$, and $w_{I(0, R)}^{z}$.
Thus it suffices to show that
\begin{align*}
(w_{I(0, \frac{1}{N}\nu^{-2b + a}), 50E}^{y}w_{I(0, \frac{1}{N}\nu^{-2b + 2a}), 50E}^{z})(R_{\ta}^{-1}T(y, z)) \lsm_{E} \wt{w}_{\Delta, 10E}(y, z).
\end{align*}
Rescaling $y$ and $z$, it is enough to prove
\begin{align}\label{weightmain3}
(w_{I(0, \frac{1}{N}\nu^{-2b + a}), 50E}^{y}w_{I(0, \frac{1}{N}\nu^{-2b + 2a}), 50E}^{z})(\nu^{-2b}R_{\ta}^{-1}T(y, z)) \lsm_{E} \wt{w}_{B(0, 1), 10E}(y, z).
\end{align}
The left hand side of \eqref{weightmain3} is equal to
\begin{align*}
(1 + |(N\nu^{-a}R_{\ta}^{-1}T(y, z))_{1}|)^{-50E}(1 + |(N\nu^{-2a}R_{\ta}^{-1}T(y, z))_{2}|)^{-50E}.
\end{align*}
We observe that
\begin{align*}
N\nu^{-2a}R_{\ta}^{-1}T &= Nd\cdot \frac{2d}{\sqrt{4d^2 + \nu^{2a}}}\begin{pmatrix}
\nu^{-a}(2 + \frac{\nu^{2a}}{2d^2}) & \nu^{-a}(3d + \frac{3}{2}\nu^{2a}d^{-1})\\
0 & 3/2
\end{pmatrix}\\
& := Nd\cdot \frac{2d}{\sqrt{4d^2 + \nu^{2a}}} S
\end{align*}
and
\begin{align}\label{ndest}
\frac{1}{\sqrt{5}}\leq Nd\cdot \frac{2}{\sqrt{5}} \leq Nd \cdot \frac{2d}{\sqrt{4d^2 + \nu^{2a}}} \leq Nd \leq 1.
\end{align}
Therefore from \eqref{ndest} and that $2 \leq 2 + \frac{\nu^{2a}}{2d^2} \leq 3$,
\begin{align*}
(1 + |(N\nu^{-a}R_{\ta}^{-1}T(y, z))_{1}|)^{-50E} &\lsm_{E} (1 + |(2 + \frac{\nu^{2a}}{2d^2})y + (3d + \frac{3}{2}\frac{\nu^{2a}}{d})z|)^{-50E}\\
&\lsm_{E} (1 + |y + d(\frac{6d^2 + 3\nu^{2a}}{4d^{2} + \nu^{2a}})z|)^{-50E}
\end{align*}
and
\begin{align*}
(1 + |(N\nu^{-2a}R_{\ta}^{-1}T(y, z))_{2}|)^{-50E} \lsm_{E} (1 + |z|)^{-50E}.
\end{align*}
Thus to prove \eqref{weightmain3}, it remains to show that
\begin{align}\label{weightbd}
\frac{(1 + |y|)}{(1 + |y + d(\frac{6d^2 + 3\nu^{2a}}{4d^{2} + \nu^{2a}})z|)^{5}(1 + |z|)^{4}}
\end{align}
is a bounded function independent of $y, z, d, \nu$, and $a$.
To see that \eqref{weightbd} is bounded, we consider the following two cases:
\begin{itemize}
\item Suppose $|y +d(\frac{6d^2 + 3\nu^{2a}}{4d^{2} + \nu^{2a}})z| \geq \frac{|y|}{2}$. Then \eqref{weightbd} is controlled by
\begin{align*}
\frac{1 + |y|}{(1 + |y|/2)^{5}(1 + |z|)^{4}} \lsm 1.
\end{align*}
\item Suppose $|y +d(\frac{6d^2 + 3\nu^{2a}}{4d^{2} + \nu^{2a}})z| < \frac{|y|}{2}$. Then $|y| \leq 2|d\frac{6d^2 + 3\nu^{2a}}{4d^{2} + \nu^{2a}}||z| \leq |z|$
and hence \eqref{weightbd} is controlled by
\begin{align*}
\frac{1}{(1 + |z|)^{3}} \lsm  1.
\end{align*}
\end{itemize}
This then proves \eqref{weightmain3} and hence also \eqref{l4main} in the case when $\nu \leq d \leq 1/1000$.
\end{proof}

\section{The iteration}\label{last_section}

We now let $C_0$ be the largest of any $C_0$ that appears in the statements of Lemmas \ref{lem3}-\ref{cor_small_ball_bi} in Section \ref{properties}.
It will no longer vary line by line as before and will now be fixed.
\begin{lemma}\label{combine}
	Let $a$ and $b$ be integers such that $1 \leq a \leq b$. Suppose $\delta$ and $\nu$ were such that $\nu^{3b}\delta^{-1} \in \N$
    and $\nu \in 2^{-2^{\mathbb{N}}} \cap (0,1/1000)$.
	Then
	\begin{align*}
	\mc{M}_{2, a, b}(\delta, \nu, E) \lsm_{\vep, E} &\nu^{\frac{1}{36}(5 + 6\vep)a - \frac{1}{36}(7 + 6\vep)b - \frac{5}{3}C_0}\times\\
&\mc{M}_{2, b, 2b - a}(\delta, \nu, E/C_0^{4})^{1/3}\mc{M}_{2, b, 3b}(\delta, \nu, E/C_0^{4})^{1/6}D(\frac{\delta}{\nu^b})^{1/2}.
	\end{align*}
\end{lemma}
\begin{proof}
	We have
	\begin{align}\label{cmbeq1}
	&\mc{M}_{2, a, b}(\delta, \nu, E)\lsm_{\vep, E} \nu^{-\frac{1}{6}(1 + \vep)(b - a)- C_0}\mc{M}_{2, 2b - a, b}(\delta, \nu, E/C_0)\nonumber\\
	&\quad\lsm_{\vep, E} \nu^{-\frac{1}{6}(1 + \vep)(b - a) - C_0} \mc{M}_{2, b, 2b - a}(\delta, \nu, E/C^2_0)^{1/3}\mc{M}_{1, 2b - a, b}(\delta, \nu, E/C^2_0)^{2/3}
	\end{align}
	where here we have used Lemmas \ref{lem3} and \ref{cor_small_ball_bi}. Next, Lemmas \ref{lem4} and \ref{bilinear1} give that
	\begin{align*}
	\mc{M}_{1, 2b - a, b}(\delta, \nu, E/C^2_0) &\lsm_{E} \nu^{-\frac{1}{24}(a + b) - C_0}\mc{M}_{1, 3b, b}(\delta, \nu, E/C_0^3)\\
	&\lsm_{E} \nu^{-\frac{1}{24}(a + b) - C_0}\mc{M}_{2, b, 3b}(\delta, \nu, E/C_0^4)^{1/4}D(\frac{\delta}{\nu^{b}})^{3/4}.
	\end{align*}
	Inserting this estimate into \eqref{cmbeq1} and observing that
	\begin{align*}
    -\frac{1}{6}(1 + \vep)( b - a) - \frac{1}{36}(a + b) = \frac{1}{36}(5 + 6\vep)a - \frac{1}{36}(7 + 6\vep)b
	\end{align*}
	then completes the proof of Lemma \ref{combine}.
\end{proof}

Let $\ld \geq 0$ be the smallest real number such that $D(\delta) \lsm_{\vep} \delta^{-1/4 - \ld - \vep}$
for all $\delta \in \N^{-1}$. The trivial bound on $D(\delta)$ shows that $\ld \leq 1/2$.
If $\ld = 0$, then we are done. We now assume $\ld > 0$ and derive a contradiction.

We will let $\wt{C}(\vep)$ be the implied constant depending on $\vep$ in the estimate $D(\delta) \lsm_{\vep} \delta^{-1/4 - \ld - \vep}$
and $C(\vep, E)$ the implied constant depending on $\vep, E$ from Lemma \ref{combine}.
\begin{lemma}\label{iter}
	Let $N \geq 0$ an integer and $\delta \in \N^{-1}$ be fixed.
	
	Suppose the following statement is true: If $b \in \N$ and $\nu \in 2^{-2^{\N}} \cap (0, 1/1000)$ is such that $\nu^{3^{N}b}\delta^{-1} \in \N$, then
	$$\mc{M}_{2, a, b}(\delta, \nu, E) \leq C_{N}(a, b, \vep, E) \delta^{-\frac{1}{4} - \ld- \vep}\nu^{-\alpha_{N}a - \beta_{N}b - \frac{10}{3}C_0}$$
	for all $a$ such that $1 \leq a \leq b$.
	
	Then the following statement is also true: If $b \in \N$ and $\nu \in 2^{-2^{\N}}\cap (0, 1/1000)$ is such that $\nu^{3^{N + 1}b}\delta^{-1} \in \N$, then
	$$\mc{M}_{2, a, b}(\delta, \nu, E) \leq C_{N + 1}(a, b, \vep, E) \delta^{-\frac{1}{4}-\ld - \vep}\nu^{-\alpha_{N + 1}a - \beta_{N + 1}b- \frac{10}{3}C_0}$$
	for all $a$ such that $1 \leq a \leq b$ where
	\begin{align*}
	\begin{pmatrix}
	\alpha_{N + 1}\\
	\beta_{N + 1}
	\end{pmatrix} =
	\begin{pmatrix}
	-5/36\\
	5/72 - \ld/2
	\end{pmatrix} +
	\begin{pmatrix}
	0 & -1/3\\
	1/2 & 7/6
	\end{pmatrix}
	\begin{pmatrix}
	\alpha_N\\
	\beta_N
	\end{pmatrix}
	\end{align*}
    and
    $$C_{N + 1}(a, b, \vep, E) := C(\vep, E)C_{N}(b, 2b - a, \vep, E/C_0^{4})^{1/3}C_{N}(b, 3b, \vep, E/C_0^{4})^{1/6}\wt{C}(\vep)^{1/2}.$$
\end{lemma}
\begin{proof}
	Fix $a$ and $b$ such that $1 \leq a \leq b$ and let $\nu \in 2^{-2^{\N}} \cap (0, 1/1000)$ be such that $\nu^{3^{N + 1}b}\delta^{-1} \in \N$.
	Then $1 \leq b \leq 2b - a$ and $$\nu^{3^{N}(2b - a)}\delta^{-1} = \nu^{3^{N + 1}b}\delta^{-1}\nu^{-3^{N}(b + a)} \in \N.$$
	Therefore by hypothesis,
	\begin{align*}
	\mc{M}_{2, b, 2b - a}(\delta, \nu, E) \leq C_{N}(b, 2b - a, \vep, E)\delta^{-\frac{1}{4} - \ld - \vep}\nu^{\beta_{N}a - (\alpha_N + 2\beta_N)b -\frac{10}{3}C_0}
	\end{align*}
	Next, since $\nu^{3^{N + 1}b}\delta^{-1} = \nu^{3^{N}(3b)}\delta^{-1} \in \N$ and $1 \leq b \leq 3b$, by hypothesis, we have
	\begin{align*}
	\mc{M}_{2, b, 3b}(\delta, \nu, E) \leq C_{N}(b, 3b, \vep, E)\delta^{-\frac{1}{4} - \ld - \vep}\nu^{-(\alpha_N + 3\beta_N)b - \frac{10}{3}C_0}.
	\end{align*}
	Finally we note that by our assumption on $D(\delta)$ and since $\nu^{3^{N + 1}b}\delta^{-1} \in \N$ implies $\nu^{b}\delta^{-1} \in \N$, we have
	\begin{align*}
	D(\frac{\delta}{\nu^{b}}) \leq \wt{C}(\vep) \delta^{-\frac{1}{4} - \ld - \vep}\nu^{b(\frac{1}{4} + \ld)}\nu^{b\vep}.
	\end{align*}
	Since $\nu^{3b}\delta^{-1} \in \N$, Lemma \ref{combine} then gives that
	\begin{align*}
	\mc{M}_{2, a, b}(\delta, \nu, E) &\leq C_{N + 1}(a, b, \vep, E) \delta^{-\frac{1}{4}- \ld- \vep}\times\\
&\quad\nu^{\frac{1}{36}(5 + 6\vep)a - \frac{1}{36}(7 + 6\vep)b - \frac{10}{3}C_0}\nu^{\frac{1}{3}\beta_{N}a - (\frac{1}{2}\alpha_N + \frac{7}{6}\beta_{N})b + \frac{1}{2}b(\frac{1}{4} + \ld)}\nu^{\frac{1}{2}b\vep}.
	\end{align*}
	Rearranging the above equation and observe that the power of $\nu^{\vep}$ is $\nu^{\frac{1}{6}\vep a + \frac{1}{3}\vep b} \leq 1$ then completes the proof of Lemma \ref{iter}.
\end{proof}

\begin{lemma}\label{linalg}
	If $\alpha_0 = 0$, $\beta_0 = 0$, and
	\begin{align}\label{linalg1}
	\begin{pmatrix}
	\alpha_{N + 1}\\
	\beta_{N + 1}
	\end{pmatrix} =
	\begin{pmatrix}
	A\\
	B
	\end{pmatrix} +
	\begin{pmatrix}
	0 & -1/3\\
	1/2 & 7/6
	\end{pmatrix}
	\begin{pmatrix}
	\alpha_N\\
	\beta_N
	\end{pmatrix},
	\end{align}
	then
	\begin{align*}
	\begin{pmatrix}
	\alpha_N\\
	\beta_N
	\end{pmatrix}=\frac{(A + 2B)N}{5}
	\begin{pmatrix}
	-1 + \frac{36}{5}(1 - \frac{1}{6^N})\frac{A}{(A + 2B)N} + \frac{12}{5}(1 - \frac{1}{6^N})\frac{B}{(A + 2B)N}\\
	3 - \frac{18}{5}(1 - \frac{1}{6^N})\frac{A}{(A + 2B)N} - \frac{6}{5}(1- \frac{1}{6^N})\frac{B}{(A + 2B)N}
	\end{pmatrix}.
	\end{align*}
\end{lemma}
\begin{proof}
	This is as in Section 4 of \cite{HB15}. Let
	\begin{align*}
	M =\begin{pmatrix}
	0 & -1/3\\
	1/2 & 7/6
	\end{pmatrix},
	\quad
	P =\begin{pmatrix}
	-1 & -2\\
	3 & 1
	\end{pmatrix},
	\quad \text{and}\quad
	D = \begin{pmatrix}
	1 & 0 \\
	0 & 1/6
	\end{pmatrix}.
	\end{align*}
	Then $M = PDP^{-1}$. Iterating \eqref{linalg1} gives
	\begin{align*}
	\begin{pmatrix}
	\alpha_N\\
	\beta_N
	\end{pmatrix} =
	&\frac{1}{5}\begin{pmatrix}
	-1 & -2\\
	3 & 1
	\end{pmatrix}
	\begin{pmatrix}
	N & 0\\
	0& \frac{6}{5}(1 - \frac{1}{6^N})
	\end{pmatrix}
	\begin{pmatrix}
	1 & 2\\
	-3 & -1
	\end{pmatrix}
	\begin{pmatrix}
	A\\
	B
	\end{pmatrix}\\
	&=\frac{(A + 2B)N}{5}
	\begin{pmatrix}
	-1 + \frac{36}{5}(1 - \frac{1}{6^N})\frac{A}{(A + 2B)N} + \frac{12}{5}(1 - \frac{1}{6^N})\frac{B}{(A + 2B)N}\\
	3 - \frac{18}{5}(1 - \frac{1}{6^N})\frac{A}{(A + 2B)N} - \frac{6}{5}(1- \frac{1}{6^N})\frac{B}{(A + 2B)N}
	\end{pmatrix}.
	\end{align*}
	This completes the proof of Lemma \ref{linalg}.
\end{proof}

We derive a contradiction.
	Setting
	$A =-5/36$ and
	$B = 5/72 - \ld/2$, we observe that $A + 2B = -\ld$.
	
	By trivially controlling the bilinear constant
	by the linear constant, if $\delta$ and $\nu$ are such that $\nu^{b}\delta^{-1} \in \N$,
	then $\mc{M}_{2, a, b}(\delta, \nu, E) \lsm_{\vep, E} \delta^{-\frac{1}{4} - \ld - \vep}$
	for all $1 \leq a \leq b$.
	Setting $\alpha_0 = 0$ and $\beta_0 = 0$ and using that $A + 2B = -\ld$, Lemma \ref{linalg}
	shows that
	\begin{align*}
	\alpha_N + \beta_N = -\frac{\ld N}{5}(2 + \frac{6}{5N}(1 - \frac{1}{6^N})(\frac{25}{72\ld} + \frac{1}{2})).
	\end{align*}
	Since $\ld > 0$, we can choose an $N_0$ sufficiently large (depending on $\ld$) such that
	$\alpha_{N_0}+ \beta_{N_0} < -1001 - \frac{10}{3}C_0$.
	Lemma \ref{iter} then shows that if $\delta \in \N^{-1}$ and $\nu \in 2^{-2^{\N}} \cap (0, 1/1000)$ are such that $\nu^{3^{N_0}}\delta^{-1} \in \N$, then
	$$\mc{M}_{2, 1, 1}(\delta, \nu, E) \lsm_{N_0, \vep, E} \delta^{-\frac{1}{4} - \ld - \vep}\nu^{1001}.$$
    Now choose $E > 1000$ to be a sufficiently large power of $C_0$ (depending on $N_0$).
	Bilinear reduction (Lemma \ref{bilinear}) then shows that if $\delta \in \N^{-1}$ and $\nu \in 2^{-2^{\N}} \cap (0, 1/1000)$ are such that $\nu^{3^{N_0}}\delta^{-1} \in \N$, we have
	\begin{align*}
	D(\delta) &\lsm_{N_0, \vep} \nu^{-\frac{1}{4}}D(\frac{\delta}{\nu}) + \delta^{-\frac{1}{4}- \ld - \vep}\nu^{1000}\\
	&\lsm_{N_0, \vep} \delta^{-\frac{1}{4} - \ld - \vep}(\nu^{\ld + \vep} + \nu^{1000})\lsm_{N_0, \vep} \delta^{-\frac{1}{4} - \ld - \vep}\nu^{\ld}
	\end{align*}
	where the last inequality is because $\ld \leq 1/2$.
	Choosing $\nu = \delta^{1/3^{N_0}}$, then shows that if $\delta$ is such that $\delta^{1/3^{N_0}} \in 2^{-2^{\N}} \cap (0, 1/1000)$, then
	\begin{align*}
	D(\delta) \lsm_{N_0, \vep} \delta^{-\frac{1}{4} -\ld(1 - \frac{1}{3^{N_0}}) - \vep}.
	\end{align*}
	Corollary \ref{almult} (almost multiplicity) then shows that $D(\delta) \lsm_{N_0, \vep} \delta^{-\frac{1}{4} -\ld(1 - \frac{1}{3^{N_0}}) - \vep}$
	for all $\delta \in \N^{-1}$.
	This contradicts the minimality of $\ld$. Therefore we cannot have $\ld > 0$ and hence we must have $\ld = 0$.
	This completes the proof of Theorem \ref{main}.

\end{document}

%% file: macros.tex
\usepackage{amsfonts,amssymb,amsmath,amsthm}
\usepackage{mathtools}
\usepackage{etoolbox}

\def\PZdefchar#1{
  \expandafter\def\csname frak#1\endcsname{\mathfrak{#1}}
  \expandafter\def\csname bf#1\endcsname{\mathbf{#1}}
  \expandafter\def\csname scr#1\endcsname{\mathcal{#1}}
  \expandafter\def\csname cal#1\endcsname{\mathcal{#1}}}
\def\PZdefloop#1{\ifx#1\PZdefloop\else\PZdefchar#1\expandafter\PZdefloop\fi}
\PZdefloop abcdefghijklmnopqrstuvwxyzABCDEFGHIJKLMNOPQRSTUVWXYZ\PZdefloop

\def\lesim{\lesssim}
\def\beq{\begin{equation}}
\def\endeq{\end{equation}}

\newcommand{\Bm}{\begin{multline}}
\newcommand{\Em}{\end{multline}}

\DeclarePairedDelimiterX\Set[1]\{\}{#1}
\DeclarePairedDelimiterXPP\EE[1]{\E}{\lparen}{\rparen}{}{#1} 

\makeatletter
\newcommand\@avsum[2]{%
  {\sbox0{$\m@th#1\sum$}%
   \vphantom{\usebox0}%
   \ooalign{%
     \hidewidth
     \smash{\vrule height\dimexpr\ht0+1pt\relax depth\dimexpr\dp0+1pt\relax}%
     \hidewidth\cr
     $\m@th#1\sum$\cr
   }%
  }%
}
\newcommand{\avsum}{\mathop{\mathpalette\@avsum\relax}\displaylimits}
\newcommand\@avprod[2]{%
  {\sbox0{$\m@th#1\prod$}%
   \vphantom{\usebox0}%
   \ooalign{%
     \hidewidth
     \smash{\vrule height\dimexpr\ht0+1pt\relax depth\dimexpr\dp0+1pt\relax}%
     \hidewidth\cr
     $\m@th#1\prod$\cr
   }%
  }%
}
\newcommand{\avprod}{\mathop{\mathpalette\@avprod\relax}\displaylimits}
\newcommand{\@avL}[2]{%
\ooalign{{$\m@th#1\mbox{--}$}\cr {$\m@th#1 L$}\cr}}
\newcommand{\avL}{\mathpalette\@avL\relax}
\newcommand{\@avell}[2]{%
\ooalign{{$\m@th#1\mbox{--}$}\cr {$\m@th#1 \ell$}\cr}}
\newcommand{\avell}{\mathpalette\@avell\relax}
\newcommand{\@avD}{%
  \ooalign{{$\mathrm{D}$}\cr \hidewidth\raise.2ex\hbox{$\vert$}\hidewidth\cr}}
\newcommand{\avDec}{\@avD\mathrm{ec}}
\makeatother

